\documentclass[a4,reqno,10pt]{amsart}

 \usepackage{amsmath,amsthm,amssymb,xcolor,graphicx,mathrsfs,tikz}

\usepackage{a4wide}
\usepackage[all]{xy}
\usepackage[colorlinks=true,
			linkcolor = black,
			urlcolor = black,
			citecolor = black]{hyperref}
\usepackage[utf8]{inputenc}
 
\newcommand{\RR}{{\mathbb{R}}}

\newcommand{\cA}{\mathcal{A}}
\newcommand{\cB}{\mathcal{B}}

\newcommand{\cF}{\mathcal{F}}
\newcommand{\cG}{\mathcal{G}}
\newcommand{\cL}{\mathcal{L}}

\newcommand{\cO}{\mathcal{O}}

\newcommand{\sfE}{\mathsf{E}}
\newcommand{\sfG}{\mathsf{G}}

\newcommand{\sfS}{\mathsf{S}}

\newcommand{\scL}{\mathscr{L}}

\newcommand{\fg}{\mathfrak{g}}

\newcommand{\fX}{\mathfrak{X}}

\newcommand{\sfGL}{\mathsf{GL}}

\newcommand{\Lie}{\mathsf{Lie}}
\newcommand{\End}{\mathsf{End}}

\DeclareMathOperator{\Ad}{Ad}

\DeclareMathOperator{\rank}{rank}

\renewcommand{\d}[1]{\ensuremath{\operatorname{d}\!{#1}{}}}
\newcommand{\s}[1]{\ensuremath{\operatorname{s}\!{#1}{}}}
\renewcommand{\t}[1]{\ensuremath{\operatorname{t}\!{#1}{}}}
\newcommand{\m}[1]{\ensuremath{\operatorname{m}{#1}{}}}
\renewcommand{\u}[1]{\ensuremath{\operatorname{u}\!{#1}{}}}
\newcommand{\ii}[1]{\ensuremath{\operatorname{i}{#1}{}}}

\newcommand{\wedgec}{{\stackrel{\wedge}{\ ,\ }}}

\theoremstyle{plain}
\newtheorem*{remark}{Remark}
\newtheorem{definition}{Definition}
\newtheorem{theorem}{Theorem}
\newtheorem{example}{Example}

\newtheorem{lemma}[theorem]{Lemma}
\newtheorem{corollary}[theorem]{Corollary}

\numberwithin{definition}{section}
\numberwithin{equation}{section}
\numberwithin{theorem}{section}
\numberwithin{example}{section}

\begin{document}


\title{Lie Algebroid Gauging of Non-linear Sigma Models}


\author[K Wright]{Kyle Wright}
\address[K Wright]{
Mathematical Sciences Institute,
Australian National University, 
Canberra, ACT 2601, Australia}
\email{kyle.wright@anu.edu.au}

\begin{abstract}
	This paper examines a proposal for gauging non-linear sigma models with respect to a Lie algebroid action.  The general conditions for gauging a non-linear sigma model with a set of involutive vector fields are given. We show that it is always possible to find a set of vector fields which will (locally) admit a Lie algebroid gauging.  Furthermore, the gauging process is not unique; if the vector fields span the tangent space of the manifold, there is a free choice of a flat connection.  Ensuring that the gauged action is equivalent to the ungauged action imposes the real constraint of the Lie algebroid gauging proposal. It does not appear possible (in general) to find a field strength term which can be added to the action via a Lagrange multiplier to impose the equivalence of the gauged and ungauged actions.  This prevents the proposal from being used to extend T-duality.  Integrability of local Lie algebroid actions to global Lie groupoid actions is discussed.
\end{abstract}

\maketitle


\tableofcontents


\section{Introduction}\label{Sec:Introduction}
Non-linear sigma models with Lie group symmetries are an important and well studied topic in string theory.  Lie algebroids are a natural generalisation of Lie algebras to vector bundles over a manifold (Lie algebras are Lie algebroids over a point).  The importance of Lie group symmetries is well established in physics.  It is natural to ask what role Lie groupoids might play in describing symmetries in non-linear sigma models.  

A gauging procedure for non-linear sigma models based on Lie algebroids has appeared in the literature \cite{Str04,Kot14,Kot15,May08}.  In principal this proposal gives a vast generalisation to the notion of gauging a non-linear sigma model.  However, there are subtleties to consider when attempting to apply this to the study of non-linear sigma models.  When this procedure was applied to the study of T-duality there was an initial belief that it let to a notion of `non-isometric T-duality' \cite{Cha16a,Cha16c}. It was later shown that gauge invariance of the action actually rendered the non-isometric T-duality proposal equivalent to non-abelian T-duality \cite{Bou17}.  The applicable Lie algebroids are in fact restricted to those which are (locally) isomorphic to Lie algebras---the existence of Killing vector fields is essential.

The above situation highlights the importance of a thorough understanding of the application of these general mathematical constructions to physical models.  This paper attempts to clarify the extent of generalisation that occurs in considering a non-linear sigma model gauged by a Lie algebroid symmetry.  The main results are the following: the generalised Killing condition (Equation \eqref{Ecompatibility}) can always be locally satisfied for an appropriate choice of vector fields (Corollary \ref{Cor:LocalGauging}). In fact, the real restriction of the proposal comes from understanding when the ungauged model can be recovered from the gauged model. It appears to be impossible (in general) to add a gauge invariant field strength term to the action to impose the equivalence of the gauged and ungauged action. This limits the usefulness of the proposal in applications where the addition of a field strength term  is necessary (such as T-duality).

The paper is organised as follows: Section \ref{Sec:Background} gives background information on Lie algebroids and groupoids; as well as describing the Lie algebroid gauging procedure and the various constraints required.  The two main constraints are closure of the gauge algebroid and a generalised Killing equation.  Section \ref{Sec:GaugeAlgebroid} takes a closer look at the closure of the gauge algebroid and shows that closure implies a flatness condition on Lie algebroid connections. These flat connections define representations of Lie algebroids describing the local gauging action.  General conditions for the solution of the generalised Killing equation (Equation \eqref{Ecompatibility}) are derived in Section \ref{Sec:GenSolution} (Theorem \eqref{Th:GaugeThrm}).  Conditions for solutions of the generalised Killing equation which also close the gauge algebroid are also derived (Theorems \eqref{Thrm:GaugeEverything} and \eqref{Thrm:GaugeArbitrary}).  Section \ref{Sec:RecoverUngauged} discusses the difficulty of imposing a condition which ensures that the original ungauged action can be recovered from the gauged action.  Comments on the extension of the local Lie algebroid actions  to global Lie groupoid actions are made in Section \ref{Sec:Groupoid}.  Flatness of algebroid connections are essential for local integrability. The fact that not all Lie algebroids are integrable (unlike Lie algebras) means that integrability is a subtle issue. Finally, the paper ends with some conclusions on Lie algebroid gauging and an outlook on the future.
 
\section{Background}\label{Sec:Background}
This section gives some relevant background material on Lie algebroids and Lie groupoids, as well as describing the Lie algebroid gauging procedure.  The two most relevant constructions are Lie algebroid connections and pullback morphisms.  The construction of flat algebroid connections is a crucial part of the Lie algebroid gauging procedure. The flat connections ${}^Q\nabla^\pm$ (introduced in Section \ref{Sec:GenSolution}) will define representations of the Lie algebroid action describing the gauging. A connection defining a Lie algebroid also gives a representation of the associated groupoid---giving the finite action of the gauging.  The description of pullback morphisms highlights a subtlety of Lie algebroids: a Lie algebroid structure is not (in general) preserved under pullbacks. There is a natural double pullback construction (see Section \ref{Sec:LieAlgMorphisms}). This is relevant to the construction of non-linear sigma models where the construction involves pulling back structure to the worldsheet.

\subsection{Lie groupoids }\label{Sec:LieAlgGroup}~\\
This section gives some relevant background information on Lie groupoids. We refer the reader to \cite{Mac05,Cra06} for a thorough introduction to the topic. In Section \ref{Sec:LieAlgGauging} Lie groupoids will be used to generalise the Lie group actions usually used to gauge non-linear sigma models. 

\begin{definition}
	A \emph{groupoid} $(\cG,M,\s{},\t{},\m{},\u{},\ii{})$ consists of a set of arrows, $\cG$, a set of objects $M$, and maps $\s{},\t{},\m{},\u{},\ii{}$, satisfying the laws of composition, associativity, and inverses:
	\begin{itemize}
		\item The \emph{source} and \emph{target} maps: $\s,\t:\ \cG\rightarrow M$, associating to each arrow $h$ its source object $\s(h)$ and target object $\t(h)$.  We write $h:x\overset{h}{\rightarrow} y$ for $h\in\cG$ satisfying $\s(h)=x$ and $\t(h)=y$.
		\item The \emph{set of composable arrows} is denoted by $\cG_2$:
		\begin{align}\label{Eq:ComposableArrows}
		\cG_2:=\{(h_2,h_1)\in\cG\times\cG: \s(h_2)=\t(h_1) \}.
		\end{align}
		For a pair of composable arrows $(h_2,h_1)$ the \emph{composition map} $\m:\cG\rightarrow \cG$ is the composition $\m(h_2,h_1)=h_2\circ h_1$ (typically denoted $h_2h_1$ for simplicity).
		\item The \emph{unit} and \emph{inverse} maps: $\u{}:  M\rightarrow \cG$, $\ii{} : \cG\rightarrow \cG$, 
		where $\u{}$ sends $x\in M$ to the identity arrow $1_x\in \cG$ at $x$, and $\ii{}$ sends an arrow $h$ to its inverse $h^{-1}$.   
	\end{itemize}
\end{definition}
For brevity we will often denote a groupoid by $(\cG,M)$ or $\cG$ if the underlying manifold is clear.

\begin{definition}
	A \emph{Lie groupoid} is a groupoid $\cG$ whose set of arrows and set of objects are both manifolds, and the structure maps $(\s,\t,\m,\u,\ii)$ are all smooth with $\s{} $ and $\t{}$ being submersions.
\end{definition}

\begin{example}[Lie group]\label{LieGroupEx}
	Every Lie group $\sfG$ can be viewed as a Lie groupoid over a point $(\cG,M)=(\sfG,\text{pt})$. 
\end{example}

\begin{example}[Fundamental Groupoid]
	Let $M$ be a manifold and let $\Pi_1(M)$
	denote the manifold consisting of all homotopy classes (with fixed end points)
	of curves in $M$. Then $(\Pi_1(M),M)$ can be endowed with the structure of a
	Lie groupoid. Let $\gamma: I\rightarrow M$ be a curve in $M$ and denote its
	homotopy class by $[\gamma]$. The source and target maps associate to $[\gamma]$ are its end points. If $\gamma_1$ and $\gamma_2$ are two curves such that
	$\gamma_1(1) = \gamma_2(0)$ then we define their product to be concatenation of curves,
	$[\gamma_2][\gamma_1] = [\gamma_2\cdot\gamma_1]$.
	The identity element at a point $x\in M$ is the class of homotopically trivial paths passing through $x$. The inverse of $[\gamma]$ is the class of $\gamma^{-1}: I\rightarrow M$, where $\gamma^{-1}(t) =\gamma(1-t)$.
\end{example}

\begin{example}[Transformation groupoid]\label{TransformationGroupoid}
	Let $\sfG$ be a group acting on a manifold $M$.  We define the \emph{transformation groupoid} $(\cG,M)=(\sfG\times M,M)$ to be the Lie groupoid whose structure maps are given by
	\begin{align*}
	\s(g,x)=&x,\quad
	\t(g,x)=gx,\quad
	(h,gx)(g,x)=(hg,x),\quad
	\u(x)=(e,x),\quad
	\ii(g,x)=(g^{-1},gx),
	\end{align*}
	where $e$ is the identity of $\sfG$.
\end{example}

\begin{example}[Gauge Groupoid]\label{GaugeGroupoid}
	Let $P(M,\pi,\sfG)$ be a principal $\sfG$-bundle.  The \emph{gauge groupoid} of $P$, denoted $(\cG(P),M)$, is defined as 
	\begin{align*}
	\cG(P)=\frac{P\times P}{\sfG},
	\end{align*}
	where the quotient refers to the diagonal action of $\sfG$ on $P\times P$ $((p,q)\cdot g=(pg,qg))$.  Let us denote by $[p,q]$ the class of $(p,q)$.  The structure of $\cG(P)$ is given by:
	\begin{align*}
	\s[p,q]=&\pi(q),\quad
	\t[p,q]=\pi(p),\quad
	[p_1,q_1][q_1,p_2]=[p_1,p_2],\\
	\u(x)=&[p,p]\text{ for some }p\in\pi^{-1}(x),\quad\ \ 
	\ii([p,q])=[q,p].
	\end{align*}
\end{example}

\begin{example}[Symplectic groupoid on $T^*\sfG$]\label{SymplecticGroupoid}
	Given a Lie group $\sfG$, with $\fg=\Lie(\sfG)$, we can define a Lie groupoid $(\cG,M)=(T^*\sfG,\fg^*)$,  equipped with its canonical symplectic structure.  If we identify $T^*\sfG\cong \sfG\times \fg^*$ the Lie groupoid structure on $T^*\sfG$ is simply that of the transformation groupoid associated to the coadjoint action of $\sfG$ on $\fg^*$.  For more details see \cite{Wei87}.
\end{example}

Lie algebras are the infinitesimal objects associated to Lie groups with $\fg=T_e\sfG$.  The Lie algebra structure can be associated to right(left)-invariant vector fields on $T\sfG$.  A similar relationship holds between Lie algebroids associated to Lie groupoids.  Given a Lie groupoid $\sfG$, we can associate a Lie algebroid to the sub-bundle $T^{\s{}} \cG=\ker(\d{\s{}})\subset T\cG$---those sections which are tangent to the $\s{}$-fibres. The Lie algebroid structure is associated to right-invariant sections on $\cG$.~\\

\subsection{Lie algebroids}\label{Sec:LieAlgebroids}~\\
A general Lie algebroid is an axiomisation of the tangent Lie algebroid of a Lie groupoid.

\begin{definition}\label{Def:LieAlgebroid}
	A \emph{Lie algebroid} $(Q,M,\rho,[\cdot\ ,\cdot]_Q)$ consists of a vector bundle $Q$ over a manifold $M$ equipped with a bundle map $\rho:Q\rightarrow TM$ and a bracket $[\cdot,\cdot]_Q:\Gamma(Q)\times \Gamma(Q)\rightarrow \Gamma(Q)$ satisfying:
	\begin{subequations}
		\begin{align}
		[q_1,[q_2,q_3]_Q]_Q=&[[q_1,q_2]_Q,q_3]_Q+[q_2,[q_1,q_3]_Q]_Q,\label{JacobiIdentity}\\
		[q_1,fq_2]_Q=&f[q_1,q_2]_Q+((\rho(q_1)f))q_2,\label{LieAlgScaling}\\
		[q_1,q_2]_Q=&-[q_2,q_1]_Q,\label{SkewSymmetry}
		\end{align}
	\end{subequations}
	for $q_1,q_2,q_3\in\Gamma(Q)$ and $f\in C^\infty(M)$.
\end{definition}
Identities \eqref{JacobiIdentity} and \eqref{SkewSymmetry} imply that $[\cdot,\cdot]_Q$ is a Lie bracket.  The first two identities imply that the anchor map $\rho$ is a bracket homomorphism:
\begin{equation}\label{anchorhomo}
\rho([q_1,q_2]_Q)=[\rho(q_1),\rho(q_2)]_{TM}
\end{equation}
where $[\cdot\ ,\cdot]_{TM}$ is the bracket defined by the commutator of vector fields.

\begin{example}[Lie algebra]\label{LieAlgebraEx}
	A Lie algebra $\fg$ defines a Lie algebroid with $Q=\fg$, $M=\text{pt}$, $\rho=0$. This is the infinitesimal object corresponding to a Lie groupoid $(\cG,M)=(\sfG,\text{pt})$, where $\fg=\Lie(\sfG)$.
\end{example}

\begin{example}[Tangent bundle]\label{TangentBundle}
	Given a manifold $M$, there is Lie algebroid on $TM$ with the bracket given by commutator of vector fields, and the anchor given by the identity map $\rho=\text{\emph{Id}}_{TM}$.
\end{example}

\begin{example}[Foliations]\label{FoliatedLiealgebroid}
	Let $\cF$ be a regular  foliation on $M$, so that $T\cF\subset TM$ is an involutive distribution of constant rank.  The distribution $T\cF$ has a Lie algebroid structure, with the bracket given by the commutator of vector fields, and the anchor given by the inclusion $i:T\cF\rightarrow TM$.   The orbits of this Lie algebroid are the leaves of $\cF$.
\end{example}

\begin{example}[Infinitesimal $\fg$-action]
	Let $\psi:\fg\rightarrow \Gamma(TM)$ be an infinitesimal action of a Lie algebra $\fg$ on $TM$. The \emph{transformation Lie algebroid} is defined on $Q=M\times \fg$, where  $\rho(x,X)=\psi(X)|_{x}$ and 
	\begin{align*}
	[X_1,X_2]_{M\times\fg}(x)=[X_1,X_2]_\fg+(\psi(X_1)\cdot X_2)(x)-(\psi(X_2)\cdot X_1)(x).
	\end{align*}
	This is the Lie algebroid associated to the transformation groupoid (Example \ref{TransformationGroupoid}).
\end{example}

\begin{example}[Atiyah algebroid]\label{AtiyahAlgebroid}
	Let $P(M,\pi,\sfG)$ be a principal $\sfG$-bundle.  The \emph{Atiyah algebroid} is defined on $Q=TP/G$ as part of the exact sequence
	\begin{equation*}
	\xymatrix{0\ar[r]& (P\times\mathfrak{g})/\sfG \ar[r]  & TP/\sfG \ar[r]^{\rho}  & TM \ar@/^1pc/[l]^\sigma \ar[r] & 0}.			
	\end{equation*}
	Sections of $Q$ are identified with right-invariant vector fields on $P$.  The bracket is given by the commutator of right-invariant vector fields; the anchor is $\pi_*|_{TP/G}:A\rightarrow TM$ induced by $\pi_*$.  A short exact sequence of vector bundles gives a short exact sequence of the $C^\infty(M)$-modules of sections:
	\begin{align*}
	\xymatrix{0\ar[r]& C^\infty(P,\fg)^\sfG  \ar[r] & \fX_\sfG(P) \ar[r]^-{\pi_*}  & \fX  \ar[r] & 0},
	\end{align*}
	where $C^\infty(P,\fg)^\sfG$ is the module of $\sfG$-equivariant smooth functions from $P$ to $\fg$. If $TP/\sfG$ is a trivial bundle there is an isomorphism $C^\infty(P,\fg)^\sfG\cong C^\infty(M,\fg)$. A choice of splitting is given by a choice of one-form connection $\sigma\in\Omega^1(M,TP/\sfG)$. The corresponding curvature $F_\sigma\in \Omega^2(M,TP/\sfG)$ is given by
	\begin{align*}
	F_\sigma(v_1,v_2)=\sigma([v_1,v_2])-[\sigma(v_1),\sigma(v_2)]_{TP/\sfG}.
	\end{align*}
	Letting $(v,\gamma)\in\Gamma(A)=\fX\oplus C^\infty(M,\fg)$ the Atiyah algebroid bracket is given by
	\begin{align*}
	[(v_1,\gamma_1),(v_2,\gamma_2)]_F=&([v_1,v_2],[\gamma_1,\gamma_2]_\fg+\nabla^\sigma_{v_1}\gamma_2-\nabla^\sigma_{v_2}\gamma_1-F_{\sigma}(v_1,v_2)).
	\end{align*}
	The Atiyah algebroid is the Lie algebroid associated to the Gauge groupoid (Example \ref{GaugeGroupoid}).
	
\end{example}

\begin{example}[Generalised Atiyah sequence]
	Locally any Lie algebroid can be associated to a generalised Atiyah sequence:
	\begin{align*}
	\xymatrix{0\ar[r]& \fg_\cO \ar[r]  & Q|_{\cO} \ar[r]^{\rho}  & T\cO \ar@/^1pc/[l]^\sigma \ar[r] & 0},
	\end{align*} 
	where $\fg_{\cO}=\ker(\rho)$ denotes the isotropy algebra of $\cO$ and $T\cO=\text{Im}(\rho)$.  
\end{example}	

\begin{example}[Poisson cotangent Lie algebroid]\label{PoissonCotangent}
	Given a Poisson manifold\footnote{A \emph{Poisson structure} is a non-degenerate $\pi\in\Gamma(\wedge^2TM)$ satisfying $[\pi,\pi]_{\text{Schouten}}=0$.} $(M,\pi)$ there is a \emph{cotangent Lie algebroid}: $Q=T^*M$, $\rho(\xi)=\pi(\xi,\cdot)$ for $\xi\in\Gamma(T^*M)$, and the bracket is given by
	\begin{align}\label{PoissonAlgebroid}
	[\xi_1,\xi_2]_{T^*M}=\cL_{\pi(\xi_1,\cdot)}\xi_2-\cL_{\pi(\xi_1,\cdot)}\xi_2-d(\pi(\xi_1,\xi_2)).
	\end{align}
	This is the Lie algebroid associated to a symplectic groupoid on $T^*M$ (Example \ref{SymplecticGroupoid}).
\end{example}

The Lie algebroids of interest for the application of Lie algebroid gauging arise from flat connections.  This involves generalising the notion of a vector bundle connection to a connection on a vector bundle endowed with a Lie algebroid structure.

\begin{definition}
	Given a Lie algebroid structure on $Q\rightarrow M$, and a vector bundle $V\rightarrow M$, a \emph{$Q$-connection on $V$} is a bilinear map $\nabla:\Gamma(Q)\times \Gamma(V)\rightarrow \Gamma(V)$ which satisfies the following conditions:
	\begin{itemize}
		\item $\nabla_{fq}v=f\nabla_q v$;
		\item $\nabla_qfv=f\nabla_q v+(\rho(q)f)v$;
	\end{itemize}
	where $f\in C^\infty(M)$, $q\in\Gamma(Q)$, and $v\in\Gamma(V)$.
\end{definition}

\begin{example}[Affine connection]
	When $Q=V=TM$ we recover the notion of an affine connection.
\end{example}

\begin{example}[Vector bundle connection]
	Given a vector bundle $V\rightarrow M$, and the Tangent Lie algebroid on $TM$, a $TM$-connection on $V$ coincides with the definition of a vector bundle connection.
\end{example}

Closure of the gauge algebroid for the Lie algebroid gauging procedure will be related to the curvature of a $Q$-connection on $Q$ (See Section \ref{Sec:GaugeAlgebroid}).  

Given a $Q$-connection on $V$, denoted by $\nabla$,  there is a natural definition of curvature $R_\nabla\in\Gamma(\wedge^2 Q^*\otimes \End(V))$:
\begin{align}\label{LieAlgCurvature}
R_\nabla(q_1,q_2)(v)=\nabla_{q_1}\nabla_{q_2}v-\nabla_{q_2}\nabla_{q_1}v-\nabla_{[q_1,q_2]_Q}v,
\end{align}
for $q_1,q_2\in\Gamma(Q)$ and $v\in\Gamma(V)$.

Given a $Q$-connection on $Q$, denoted $\nabla$, we define the Lie algebroid torsion $T_\nabla\in\Gamma(\wedge^2Q^*\otimes Q)$ as
\begin{align}
T_\nabla(q_1,q_2)=\nabla_{q_1}q_2-\nabla_{q_2}q_1-[q_1,q_2]_Q.
\end{align}

\begin{definition}
	A \emph{representation} of a Lie algebroid $Q$ over $M$ on a vector bundle $V\rightarrow M$ is a choice of flat connection $\nabla:\Gamma(Q)\times \Gamma(V)\rightarrow \Gamma(V)$.
\end{definition}

Given a representation of a Lie algebroid (a choice of flat connection $\nabla:\Gamma(Q)\times\Gamma(V)\rightarrow \Gamma(V)$) we can (locally) construct a representation of the associated groupoid. A choice of connection $\nabla$ can be used to construct a path groupoid via an exponential map $\text{Exp}_\nabla:Q\rightarrow \cG(Q)$. This path groupoid is sometimes referred to as the Weinstein groupoid. Any Lie algebroid is integrable to a local Lie groupoid. Details on this construction as well as obstructions to global integrability can be found in \cite{Cra05}.

\subsubsection{\bf{Lie algebroid morphisms}}\label{Sec:LieAlgMorphisms}~\\
The following description of Lie algebroid morphisms is based on \cite{Hig88}.  Let $V\rightarrow M$ and $Q\rightarrow M$ be vector bundles over the same base manifold $M$. A morphism $\Phi:V\rightarrow Q$ induces a map of smooth sections $\Gamma(V)\rightarrow \Gamma(Q)$, given by $v\rightarrow \Phi\circ v$, which is a linear map of $C^\infty(M)$-modules. In this case the notion of a Lie algebroid morphism is equivalent to the existence of a morphism $\Phi$ between the vector bundles preserving the bracket
\begin{align*}
\Phi\circ [v_1,v_2]_{V}=[\Phi\circ v_1,\Phi\circ v_2]_{Q},
\end{align*}
for $v_1,v_2\in\Gamma(V)$.  However, for vector bundles $V\rightarrow \Sigma$ and $Q\rightarrow M$ over different bases, a morphism $\Phi:V\rightarrow Q$ and $X:\Sigma\rightarrow M$, does not induce a map between modules of sections.  It is necessary to consider the pullback bundle $X^*Q$. Sections $v\in\Gamma(V)$ can be pushed forward to sections $\Phi(v)\in\Gamma(X^*Q)$, and sections $q\in\Gamma(Q)$ can be pulled back to $X^*q\in\Gamma(X^*Q)$. Given a section $v\in\Gamma(V)$, there is a decomposition
\begin{align*}
\Phi^*(v)=f^i\otimes q_i,
\end{align*}
for some suitable $f^i\in C^\infty(\Sigma)$, and $q_i\in \Gamma(Q)$.  However this decomposition is not unique.  Choose a connection $\nabla:\Gamma(TM)\times\Gamma(Q)\rightarrow \Gamma(Q)$ and define $\bar{\nabla}:\Gamma(T\Sigma)\times \Gamma(X^*Q)\rightarrow\Gamma(X^*Q)$ by
\begin{align}\label{PullbackConnection}
\bar{\nabla}_\sigma(f_i\otimes q^i)=\sigma(f_i)\otimes q^i+f_i\otimes \nabla_{X_*(\sigma)}q^i,
\end{align}
where $\sigma\in \Gamma(T\Sigma)$, $f_i\in C^\infty(\Sigma)$, and $q^i\in \Gamma(Q)$. The definition of $\bar{\nabla}$ is not dependent on the choice of decomposition. Define the torsion of the pullback connection $\bar{\nabla}$ by
\begin{align*}
T_{\bar{\nabla}}(f_i\otimes q^i,f'_j\otimes q'^j):=f_if'_jX^*T_\nabla(q_i,q_j),
\end{align*}
and the field strength $F_\Phi\in\Gamma(\wedge^2V^*\otimes X^*Q)$ by
\begin{align}
F_\Phi(v_1,v_2)=\bar{\nabla}_{\rho(v_1)}\Phi(v_2)-\bar{\nabla}_{\rho(v_2)}\Phi(v_1)-\Phi([v_1,v_2]_V)-T_{\nabla}(\Phi(v_1),\Phi(v_2)).
\end{align}
The definition of $F_\Phi$ is independent of the choice of connection $\nabla$.  The pair $(\Phi,X)$ defines a \emph{Lie algebroid morphism} if and only if $F_\Phi\equiv 0$.    

\subsubsection{\bf{Example: Double pullback of a Lie algebroid}}\label{Sec:Pullback}~\\
In this subsection we consider the pullback of a Lie algebroid structure.  This will be of particular interest in Section \ref{Sec:GaugingDescription} when discussing Lie algebroid gauging of non-linear sigma models.

Consider a Lie algebroid $Q\rightarrow M$, and a smooth map $X:\Sigma\rightarrow M$.  There is no natural induced Lie algebroid on $X^*Q$, due to the fact that a vector bundle morphism does not induce a map between the modules of sections.  However, there may be an induced Lie algebroid structure.  This construction is due Higgins and Mackenzie \cite{Hig88}.

Consider the following bundle map:  
\begin{equation*}
\begin{tikzpicture}[node distance=2cm, auto]
\node (node) {};
\node [below of=node] (TM1) {$T\Sigma$};
\node [right of=TM1, node distance =3cm] (TM2) {$TM$};
\node [right of=node, node distance =3cm] (B) {$Q$};

\draw[->] (TM1) -- node[above,anchor=south]{$X_*$} (TM2); 
\draw[->] (B) -- node[above, anchor=west]{$\rho$} (TM2);
\end{tikzpicture}
\end{equation*}
The aim is to construct a Lie algebroid structure on $(X^{**}Q)\rightarrow \Sigma$, using the bundle maps 
\begin{equation*}
\begin{tikzpicture}[node distance=2cm, auto]
\node (A) {$X^{**}Q$};
\node [below of=node] (TM1) {$T\Sigma$};
\node [right of=TM1, node distance =3cm] (TM2) {$X^*TM$};
\node [right of=node, node distance =3cm] (B) {$X^*Q$};

\draw[->] (TM1) -- node[above,anchor=south]{$(X_*)^*$} (TM2); 
\draw[->] (B) -- node[above, anchor=west]{$\phi^*\rho$} (TM2);
\draw[->] (A) -- node[above, anchor=east]{$\check{\rho}$} (TM1);
\draw[->] (A) -- node[above, anchor=south]{$\Phi^*$} (B);
\end{tikzpicture}
\end{equation*}
Sections of $X^{**}Q$ are of the form $s\oplus \beta$, where $s\in\Gamma(T\Sigma)$ and $\beta\in\Gamma(X^*Q)$.  The induced Lie algebroid $X^{**}Q$ exists whenever $X$ is a surjective submersion, or $Q$ is transitive, or if $X_*$ and $\rho$ are transversal. In these cases there is a Lie algebroid structure on $X^{**}Q$ described as follows: define $\check{\rho}(\sigma\oplus \beta)=\sigma$ and 
\begin{align}\label{PullbackBracket}
[s_1\oplus \beta_1,s_2\oplus \beta_2]_A:=[s_1,s_2]_{T\Sigma}\oplus \Big(\bar{\nabla}_{s_1}\beta_2-\bar{\nabla}_{s_2}\beta_1-T_{\bar{\nabla}}(\beta_1,\beta_2) \Big),
\end{align}
for $\sigma_1,\sigma_2\in\Gamma(T\Sigma)$ and $\beta_1,\beta_2\in\Gamma(X^*Q)$.

\begin{example}
	Take $M$ to be a point, and hence $Q=\fg$ is a Lie algebra. The inverse image connection in $\Sigma\times \fg$ is the standard flat connection $\check{\nabla}_{s}(x)=s(x)$, and the above formula reduces to the standard expression of the bracket in $T\Sigma\oplus(\Sigma\times \fg)$.
\end{example}

It is of interest to consider the existence of the groupoid $X^{**}\cG$.  Given a Lie groupoid $\cG$ on $M$ and a smooth map $X:\Sigma\rightarrow M$, such that $X\times X:\Sigma\times \Sigma\rightarrow M\times M$ and $(\s,\t):\cG\rightarrow M\times M$ are transversal, the pullback 
\begin{equation*}
\begin{tikzpicture}[node distance=2cm, auto]
\node (A) {$X^{**}\cG$};
\node [below of=node] (M1) {$\Sigma\times \Sigma$};
\node [right of=M1, node distance =3.5cm] (M2) {$M\times M$};
\node [right of=node, node distance =3.5cm] (B) {$\cG$};

\draw[->] (M1) -- node[anchor=north]{$X\times X$} (M2); 
\draw[->] (B) -- node[above, anchor=west]{$(\s,\t)$} (M2);
\draw[->] (A) -- node[above, anchor=east]{} (M1);
\draw[->] (A) -- node[above, anchor=south]{} (B);
\end{tikzpicture}
\end{equation*}
is a manifold and has a groupoid structure.  However, it is not necessarily true that the source or target maps are submersions, meaning $X^{**}\cG$ may not be a Lie groupoid.  If the composition
\begin{align*}
\xymatrix{
	X^*\cG \ar[r] & \cG \ar[r]^{\t{}} & M,
}
\end{align*}
is a submersion then $X^{**}\cG$ will be a Lie groupoid.  Here $X^*\cG$ denotes the pullback of $X$ and $\s{}$:
\begin{align*}
X^*\cG:=\{(\fX,\sigma)\in \cG\times \Sigma: \s(\fX)=X(\sigma) \}.
\end{align*}
The double pullback construction will be essential in defining Lie algebroid gauging in Section \ref{Sec:LieAlgGauging}.

\subsection{Lie algebroid gauging}\label{Sec:GaugingDescription}~\\
This section describes the procedure for gauging a non-linear sigma model with respect to a set of involutive vector fields $\{\rho_a \}$ and two vector bundle connections $\nabla^\pm$ specified by connection coefficients $(\omega\pm\phi)^a{}_{\mu b}$. 

\subsubsection{\bf{Kotov--Strobl Lie algebroid gauging}}\label{sec:KotovStrobl}~\\
This section outlines the local coordinate description of Lie algebroid gauging developed by Kotov and Strobl \cite{Str04,Kot14}, Mayer and Strobl \cite{May08}, and further studied with Chatzistavrakidis, Deser, and Jonke \cite{Cha16b}.

The general proposal for Lie algebroid gauging can be found in \cite{Cha16b}.  Consider a map $X:\Sigma\rightarrow M$, embedding a string worldsheet into a target space $M$.  This map can be described locally by $X^\mu$, for $\mu=1,\dots,\dim(M)$.  The key generalisation associated to Lie algebroid gauging is the ability to gauge with respect to a set of involutive vector fields $\rho_a\in TM$, $a=1,\dots,d$ satisfying\footnote{The integer $d$ need not be equal $\dim(M)$.} 
\begin{align}\label{StructureFunctions}
[\rho_a,\rho_b]=C^c{}_{ab}(X)\rho_c,\quad C^c{}_{ab}(X)\in C^\infty(M),
\end{align}  
defining a Lie algebroid.  A Lie algebroid structure can be defined as follows:  Let $Q\rightarrow M$ be a vector bundle, specified locally by a frame $\{ e_a \}$, $a=1,\dots,d=\dim(Q)$, satisfying
\begin{align}\label{BasicQLiealgebroid}
[e_a,e_b]_Q:=C^c{}_{ab}(X)e_c.
\end{align}  
The anchor $\rho:Q\rightarrow TM$ is defined by $\rho(e_a):=\rho_a$.  The gauged action given in \cite{Cha16b} is
\begin{align}\label{KSaction}
S_{\text{KS}}[X,A]=\frac{1}{2}\int_{\Sigma}G_{\mu\nu}DX^\mu\wedge \star DX^\nu+\int_{\Sigma_3}H+\int_{\Sigma}(A^a\wedge\alpha_a+\frac{1}{2}\gamma_{ab}A^a\wedge A^b),
\end{align} 
where $DX^\mu:=\d X^\mu-\rho^\mu_aA^a$, $H\in\Omega^3(M)$, $\Sigma_3$ is a three manifold with boundary $\Sigma$, $A\in\Omega^1(\Sigma,X^*Q)$,  $\alpha\in\Gamma(Q^*)$, and $\gamma\in \Gamma(\wedge^2Q^*)$. 

The infinitesimal gauge transformations are of the form
\begin{subequations}\label{MSgauge}
	\begin{align}
	\delta_\varepsilon X^\mu=&\rho^\mu_a(X)\varepsilon^a\\
	\delta_{\varepsilon}A^a=&\d \varepsilon^a+C^a{}_{bc}(X)A^b\varepsilon^c+\omega^a{}_{\mu b}(X)\varepsilon^bDX^\mu+\phi^a{}_{\mu b}(X)\varepsilon^b\star DX^\mu,
	\end{align}
\end{subequations}
where $\omega^a{}_{\mu b},\phi^a{}_{\mu b}\in C^\infty(M)$ are a priori undetermined fields and $\star$ denotes the Hodge star on the worldsheet.  

Under a change of frame $\tilde{e}_a=K^b{}_{a}e_b$ the fields $\omega^a{}_{\mu b}$ and $\phi^a{}_{\mu b}$ transform as
\begin{align*}
\tilde{\omega}^a{}_{\mu b}=(K^{-1})^a{}_c\omega^c{}_{\mu d}K^d{}_{b}-K^c{}_{b}\partial_\mu (K^{-1})^a{}_{c},\quad \tilde{\phi}^a{}_{\mu b}= (K^{-1})^a{}_{c}\phi^c{}_{\mu d}K^d{}_b.
\end{align*}
Thus $\omega:\Gamma(Q)\rightarrow \Gamma(Q\otimes T^*M)$ defines a $TM$-connection on $Q$, and $\phi\in\Omega^1(M,\End(Q))$.

The action $S_{\text{KS}}[X,A]$ is invariant under the gauge transformations \eqref{MSgauge} if the following constraints hold:
\begin{subequations}\label{KSconstraints}
	\begin{align}
	\cL_{\rho_a}G=&\omega^b{}_a\vee \iota_{\rho_{b}}G+\phi^b{}_a\vee\alpha_b ,\\
	\cL_{\rho_a}H=&\d \alpha_a-\omega^b{}_a\wedge \alpha_b\pm\phi^b{}_a\iota_{\rho_b}G,\\
	\gamma_{ab}=&\iota_{\rho_a}\alpha_b,\\
	\cL_{\rho_a}\alpha_b=&C^c{}_{ab}\alpha_c+\iota_{\rho_b}(\d \alpha_a-\iota_{\rho_a}H),
	\end{align}
\end{subequations}
where $(\omega^b{}_a\vee \iota_{\rho_b}G)_{\mu\nu}=\omega^b{}_{\mu a}\rho^\lambda_bG_{\lambda\nu}+\omega^b{}_{\nu a}\rho^\lambda_b G_{\mu\lambda}$, and the choice
$\pm$ is given by the choice of Lorentzian ($\star^2=1$) or Euclidean ($\star^2=-1$) signature on the worldsheet.

Closure of the gauge algebroid (defined by Equations \eqref{MSgauge}) and the required constraints \eqref{KSconstraints} form a formidable set of conditions.  There are two natural questions:
\begin{enumerate}
	\item For a given choice of $G$ and $H$, do there exist  $(\rho_a,\alpha_a,\omega^a{}_{\mu b},\phi^a{}_{\mu b})$ satisfying the constraints \eqref{KSconstraints} which also result in a closed gauge algebroid (given by \eqref{MSgauge})?
	\item If an appropriate choice of $(\rho_a,\alpha_a,\omega^a{}_{\mu b},\phi^a{}_{\mu b})$ exists is the choice  unique?  
\end{enumerate}  
An answer to the existence question is given for special cases in \cite{Cha16b,Cha17}.  The results of this paper give a more complete answer: Corollary \ref{Cor:LocalGauging} states that for any choice of $G$ and $H$ there exist $(\rho_a,\alpha_a,\omega^a{}_{\mu b},\phi^a{}_{\mu b})$ which satisfy the constraints \eqref{KSconstraints} for some $U\subset M$.  Necessary and sufficient conditions to gauge with respect to a chosen set of vector fields $\rho_a\in\Gamma(TM)$ are determined (Theorem \ref{Thrm:GaugeEverything} and Theorem \ref{Thrm:GaugeArbitrary}).  If $\rho_a\in\Gamma(TM)$ do satisfy the necessary and sufficient conditions a (not necessarily unique) solution for $(\alpha_a,\omega^a{}_{\mu b},\phi^a{}_{\mu b})$ is given. 

\subsubsection{\bf{Pullback constraint of Kotov--Strobl gauging}}\label{PullbackKS}~\\
It should be noted that there are implied pullbacks in the action and the gauge transformations \eqref{MSgauge}.   Closure of the gauge algebroid requires that the Lie algebroid structure can be pulled back.  In general this is not possible. This represents a serious restriction on the allowable Lie algebroids described by this method (see Section \ref{PullbackKS}).

The Lie algebroid described by Equation \eqref{BasicQLiealgebroid} is not invariant. Given a change of frame $\tilde{e}_a=K^b{}_ae_b$, for $K\in C^\infty(M,\sfGL(d))$, the description becomes $[\tilde{e}_a,\tilde{e}_b]_Q=\widetilde{C}^c{}_{ab}(X)\tilde{e}_c$,
where
\begin{align}
\widetilde{C}^c{}_{ab}=(K^{-1})^c{}_{d}(K^e{}_{a}K^f{}_{b}&C^d{}_{ef}+K^e{}_{a}\rho^\mu_e\partial_\mu K^d{}_{b}-K^e{}_{b}\rho^\mu_e\partial_\mu K^d{}_{a}).\label{StructureChangeFrame}
\end{align}

The variation of the gauge fields for the Kotov--Strobl gauging proposal are given by Equations \eqref{MSgauge}.  The closure of the gauge algebroid requires that
\begin{align*}
[\delta_{\varepsilon_1},\delta_{\varepsilon_2}]X^\mu=\delta_{\varepsilon_3}X^\mu,\quad [\delta_{\varepsilon_1},\delta_{\varepsilon_2}]A^a=\delta_{\varepsilon_3}A^a,
\end{align*}
for some $\varepsilon_3=\sigma(\varepsilon_1,\varepsilon_2)=-\sigma(\varepsilon_2,\varepsilon_1)\in \Gamma(X^*Q)$.  The field $\varepsilon_3\in\Gamma(X^*Q)$ can be written using the pullback of a basis for $Q$: $\varepsilon_3=\varepsilon^a_3X^*e_a\in \Gamma(X^*Q)$.  The following expression for $\varepsilon^a_3$ is given in the literature (Equation (10) in \cite{May08}):
\begin{align}\label{KSclosure}
\varepsilon^a_3=\varepsilon^b_1\varepsilon^c_2C^a{}_{bc}.
\end{align} 
This identification only allows for a restricted set of Lie algebroid structures.  We can see this explicitly. Consider the Lie algebroid structure $(Q,[\cdot,\cdot]_Q,\rho)$ (Defined by \eqref{BasicQLiealgebroid}) restricted to the image $(X(\Sigma),Q|_{X(\Sigma)})\subset (M,Q)$. Denote the restricted algebroid structure  $(Q|_{X(\Sigma)},[\cdot,\cdot]_{X(\Sigma)},\rho_{X(\Sigma)})$.  Take a change of frame on $Q|_{X(\Sigma)}$ given by
$\tilde{e}_a=K(X(\sigma))^b{}_{a}e_b$.
Invariance of sections gives the transformation of the coefficients: 
\begin{align*} \varepsilon=\varepsilon^aX^*e_a=\tilde{\varepsilon}^aX^*\tilde{e}_a=\tilde{\varepsilon}^aK(\sigma)^b{}_{a}X^*e_b,\quad\Longrightarrow \quad  \tilde{\varepsilon}^a=(K^{-1})^a{}_b\varepsilon^b.
\end{align*} 
This gives a constraint on the transformation of the structure functions on $X(\Sigma)$:
\begin{align*}
\varepsilon_3=\varepsilon^b_1\varepsilon^c_2C^a{}_{bc}X^*e_a=\tilde{\varepsilon}^b_1\tilde{\varepsilon}^c_2\widetilde{C}^a{}_{bc}X^*\tilde{e}_a=\varepsilon^y_1\varepsilon^z_2(K^{-1})^b{}_{y}(K^{-1})^c{}_{z}\widetilde{C}^x_{yz}K^a{}_xX^*e_a,
\end{align*}
and we conclude that
\begin{align}
\widetilde{C}^a{}_{bc}=(K^{-1})^a{}_{x}C^x{}_{yz}K^y{}_b K^z{}_c.\label{Cconstraint1}
\end{align}
However, it follows from \eqref{StructureChangeFrame} that the structure functions (restricted to $X(\Sigma)$) transform as
\begin{align}
\widetilde{C}^c{}_{ab}=(K^{-1})^c{}_{d}(K^e{}_{a}K^f{}_{b}&C^d{}_{ef}+K^e{}_{a}(X^{-1}_*\rho)_e( K^d{}_{b})-K^e{}_{b}(X^{-1}_*\rho)_e(K^d{}_{a})),\label{Cconstraint2}
\end{align}
where $(X^{-1}_*\rho)_a$ denotes the pushforward of the map $X^{-1}$ (which exists as $X$ is a diffeomorphism when restricted to $X(\Sigma)\subset M$).  It is clear that the requirement \eqref{Cconstraint1} places a tight constraint on the allowable Lie algebroids for gauging.  In particular, the requirement that \eqref{Cconstraint1} and \eqref{Cconstraint2} hold simultaneously, mean that the Lie algebroid bracket $[\cdot,\cdot]_{X(\Sigma)}$ is $C^\infty(X(\Sigma))$ linear.

The Lie algebroid gauging procedure outlined by Kotov, Mayer, Strobl and CDJ \cite{Str04,May08,Kot14,Kot15,Cha16b} is only valid when $(Q|_{X(\Sigma)},[\cdot,\cdot]_{X(\Sigma)},\rho_{X(\Sigma)})$---the restriction of $(Q,[\cdot,\cdot]_Q,\rho)$ to the image of $X$---is a bundle of Lie algebras. 

\subsubsection{\bf{Lie algebroid gauging}}\label{Sec:LieAlgGauging}~\\
This section describes a general construction for considering non-linear sigma models which are gauged with respect to a Lie algebroid action.  This construction is valid for any integrable Lie algebroid, $Q\cong\Lie(\cG)$.  

A two-dimensional non-linear sigma model Consists of the data $(X,\Sigma,h,M,G,H,S[X])$:
where $X:\Sigma\rightarrow M$ describes the embedding of a two-dimensional (pseudo-)Riemannian surface $(\Sigma,h)$ (the string worldsheet) in an $n$-dimensional (pseudo-)Riemannian manifold $(M,G)$ (the target space); the dynamics of the string are encoded in an action
\begin{align}
S[X]=\frac{1}{2}\int (X^*G)_{\mu\nu}\d X^\mu\wedge\star \d X^\nu+\int_{\Sigma_3}X^*H,
\end{align}
where $\star$ is the Hodge star of the worldsheet, $\Sigma_3$ is a three-dimensional manifold with boundary $\Sigma$, and $H\in\Omega^3(M)$.

The gauged model takes a non-linear sigma model $(X,\Sigma,h,M,G,S[X])$ and constructs an action $S_Q[X,A]$ for some choice of $\{\rho_a\}$ and $\{(\Omega^\pm)^a{}_{\mu b}\}$, where the vector fields $\rho_a\in\Gamma(TM)$ define a Lie algebroid
\begin{align*}
[\rho_a,\rho_b]=:C^c{}_{ab}\rho_c,\quad C^c{}_{ab}\in C^\infty(M),
\end{align*}
and $(\Omega^\pm)^a{}_{\mu b}$ define $TM$-connections on a vector bundle $Q$. The $TM$-connections on $Q$ will be denoted $\nabla^\pm$.  Let $\{e_a \}$ be a local frame for $Q$.  The connections $\nabla^\pm$ are defined as
\begin{align}
\nabla^\pm e_a:=(\Omega^\pm)^b{}_{a}\otimes e_b.
\end{align} 
The fields $(\rho_a,(\Omega^\pm)^a{}_{\mu b})$ determine $(\check{\rho}_a,\check{\Omega}^a{}_{\alpha b})$ which are defined in Section \ref{ProperPullback} using the pullback construction.

Let $\{\sigma^\alpha\}$ be local coordinates on the worldsheet $\Sigma$.  Consider the following action (for $\star^2=1$ Lorentzian worldsheet):
\begin{equation}\label{GenLieAction}
S_Q[X,A]=\int_{\Sigma}(X^{**}E)_{\alpha\beta}D_-\sigma^\alpha\wedge D_+\sigma^\beta-\int_{\Sigma}X^{**}C+\int_{\Sigma_3}X^{**}H,
\end{equation}
where, $E=G+C$ for some $C\in\Omega^2(M)$, $D\sigma=d\sigma-\check{\rho}(A)$, $A\in\Omega^1(\Sigma,X^{**}Q)$, 
and $D_\pm \sigma=\frac{1}{2}(D\sigma\pm\star D\sigma)$. Remembering that $\Gamma(X^{**}Q)\cong \Gamma(T\Sigma)\oplus \Gamma(X^*Q)$ we define $X^{**}E\in\Gamma(X^{**}Q^*\otimes X^{**}Q^*)$ and $X^{**}C\in\Omega^2(\Sigma,X^{**}Q)$ as
$X^{**}E=\left(\begin{smallmatrix}
h & 0\\ 0 &X^*E
\end{smallmatrix}\right)$ and 
 $X^{**}C=\left(\begin{smallmatrix}
 0 & 0\\ 0 &X^*C
 \end{smallmatrix}\right)$ using the decomposition $X^{**}Q^*\cong T^*\Sigma\oplus X^*Q^*$. The 3-form  $X^{**}H\in\Omega^3(\Sigma,X^{**}Q)$ has $X^{*}H\in\Omega^3(\Sigma,X^*Q)$ as the only non-zero component.

The infinitesimal variation is given by 
\begin{subequations}\label{GaugeAlg+}
	\begin{align}
	\delta_{\varepsilon} X=&\check{\rho}(\varepsilon)(X),\\
	\delta_{\varepsilon}A^i=&d\varepsilon^i+\check{C}^i{}_{jk}A^j\varepsilon^k+(\check{\Omega}^+)^i{}_{\alpha j}\varepsilon^jD_+\sigma^\alpha+(\check{\Omega}^-)^i{}_{\alpha j}\varepsilon^jD_-\sigma^\alpha,
	\end{align}
\end{subequations}
and $\delta_{\varepsilon}S_Q[X,A]=0$ if the following conditions are met:
	\begin{align}
	(\cL_{\rho_a}E)_{\mu\nu}=&E_{\mu\lambda}\rho^\lambda_b(\Omega^+)^b{}_{\nu a}+E_{\lambda\nu}\rho^\lambda_b(\Omega^-)^b{}_{\mu a},\label{Ecompatibility}\\
	\cL_{\rho(\varepsilon)}C=&\iota_{\rho(\varepsilon)}H.\label{Eq:DefineC}
	\end{align}
We will refer to \eqref{Ecompatibility} as the \emph{generalised Killing equation}.

We note that the infinitesimal variations \eqref{GaugeAlg+} are equivalent (up to pullbacks) to \eqref{MSgauge} if we make the identification $(\Omega^\pm)^a{}_{\mu b}=(\omega \pm\phi)^a{}_{\mu b}$.

\subsubsection{\bf{Gauge algebroid}}\label{ProperPullback}~\\
It was shown in Section \ref{PullbackKS} that taking $\varepsilon\in\Gamma(X^*Q)$ and setting
$\delta_\varepsilon X^\mu=X^*\rho(\varepsilon)(X^\mu)$
leads to a strong restriction on the Lie algebroids that can be used to gauge.  This is a consequence of the fact that a Lie algebroid structure does not naturally pullback in general.  However, there may be a natural Lie algebroid structure defined on $X^{**}Q$.  The induced Lie algebroid $X^{**}Q$ always exists when considering Lie algebroid gauging as $X:\Sigma \rightarrow M$ is an embedding.  If there is a Lie groupoid $\cG(Q)$ such that $Q=\Lie(\cG)$ then $\cG(X^{**}Q)$ will give a well defined Lie groupoid.

The pullback Lie algebroid is essential in defining the gauge algebroid used in the Lie algebroid gauging procedure.  The Lie algebroid bracket $[\cdot,\cdot]_{X^{**}Q}$ (defined via \eqref{PullbackBracket}) is induced from $[\cdot,\cdot]_Q$ and determines the Lie algebroid structure functions $\check{C}^i{}_{jk}\in C^\infty(\Sigma)$:
\begin{align*} [\check{e}_i,\check{e}_j]_{X^{**}Q}:=\check{C}^k{}_{ij}\check{e}_k,
\end{align*}
where $i,j,k=1,\dots,\dim(X^{**}Q)$ and $\{\check{e}_i \}$ gives a local frame for $X^{**}Q$.  Sections of $X^{**}Q$ are given by $\varepsilon=(\sigma,\epsilon)\in \Gamma(T\Sigma\oplus X^*Q)$.  The anchor $\check{\rho}:X^{**}Q\rightarrow T\Sigma$ is given by
\begin{align*}
\check{\rho}(\varepsilon)=\check{\rho}(\sigma,\epsilon):=\sigma.
\end{align*}
The variation $\delta_\varepsilon X=\check{\rho}(\varepsilon)X$ is now a well defined quantity, and 
\begin{align*}
[\delta_{\varepsilon_1},\delta_{\varepsilon_2}]X=\delta_{[\varepsilon_1,\varepsilon_2]}X\quad \Longrightarrow\quad [\check{\rho}(\varepsilon_1),\check{\rho}(\varepsilon_2)]_{T\Sigma}X=\check{\rho}([\varepsilon_1,\varepsilon_2]_{X^{**}Q})X.
\end{align*}
We conclude that $[\delta_{\varepsilon_1},\delta_{\varepsilon_2}](X)=\delta_{[\varepsilon_1,\varepsilon_2]}(X)$ is equivalent to the anchor homomorphism property of the Lie algebroid $(X^{**}Q,[\cdot,\cdot]_{X^{**}Q},\check{\rho})$. The variation $\delta_\varepsilon X$ is generated by infinitesimal diffeomorphisms on $\Sigma$ generated by $\check{\rho}(\varepsilon)\in\Gamma(T\Sigma)$.

The fields $(\Omega^\pm)^a{}_{\mu b}$ ($a,b=1,\dots,\dim(Q)$ and $\mu=1,\dots,\dim(M)$) define connections via $\nabla^\pm e_a:=(\Omega^\pm)^b{}_{a}\otimes e_b$.  We define $(X^*\Omega^\pm)^a{}_{\alpha b}$ by writing the connections $X^*\nabla^\pm$ (defined via \eqref{PullbackConnection}) on the basis $\{ X^*e_a\}$. We define a $T\Sigma$-connection on $X^{**}Q$ as follows:
\begin{align}
\check{\nabla}^\pm_{\sigma'} \sigma\oplus\epsilon:=\rho(\sigma')(\sigma)\oplus\epsilon+\sigma\oplus X^*\nabla^\pm_{\sigma'}\epsilon.
\end{align}
On the basis $\check{e}_i=\partial_\alpha\oplus X^*e_a$, we get the connection coefficients:
\begin{align*}
\check{\nabla}^\pm\check{e}_i=(\check{\Omega}^
\pm)^j{}_{i}\otimes \check{e}_j.
\end{align*}

With these definitions the variation of the gauge fields $A\in \Omega^1(\Sigma,X^{**}Q)$ is a well defined quantity. 

\begin{remark}
	Having established the correct gauge algebroid structure, we will henceforth omit the $\check{\cdot}$ with the relevant double pullback maps implied.  This will allow a more direct comparison to the formulae appearing in the literature on Lie algebroid gauging.\footnote{It will be seen in Section \ref{Sec:GaugeAlgebroid} that the main constraint of the construction (closure of the gauge algebroid) implies that the tensorial quantities $R_{{}^{X^{**}Q}\nabla^\pm}=0$.  The results of this paper can be stated in terms of structures on $Q$ after noting that $X^{**}(R_{{}^{Q}\nabla^\pm})=R_{{}^{X^{**}Q}\nabla^\pm}$.}~\\
\end{remark}
\subsection{Examples}~\\
The action $S_Q[X,A]$ includes the non-linear sigma models described by $S_{\text{KS}}[X,A]$ and $S_{\text{WZW}}[X,A]$ as well as allowing the possibility of gauging the standard non-linear sigma model with any integrable Lie algebroid.

\begin{example}[WZW]
	$S_{\text{WZW}}[X,A]$ can be described using $S_Q[X,A]$ by taking $M=\sfG$, $Q=T\sfG=\sfG\times \fg$. Choose the coframe $\d X=\eta=g^{-1}\d g$ and identify    
	\begin{align*}
	(E,\rho,C,H)=((G+B)_{\mu\nu}\eta^\mu\otimes \eta^\nu,\Ad_{g^{-1}},B,H),
	\end{align*}
	where $H=(g^{-1}\d g,[g^{-1}\d g,g^{-1}\d g]_\fg)_\sfG$ is given by $H=\d{}_\fg B$.
\end{example}
\begin{example}[Poisson sigma model]\label{PoissonModel}
	Define $B_{\mu\nu}:=A_\mu (\pi(A,\cdot))^{-1}_\nu$ where $\pi$ is a Poisson bivector; the action is given by 
	\begin{align*}
	S[X,A]=\frac{1}{2}\int_{\Sigma}G_{\mu\nu}DX^\mu\wedge \star DX^\nu+\int_{\Sigma}A_\mu\wedge \d X^\mu+\frac{1}{2}\pi^{\mu\nu}A_\mu\wedge A_\nu,
	\end{align*}
	where $DX=\d X-\pi(A,\cdot)$.  The Lie algebroid structure on $Q=T^*M$ is given by the Poisson cotangent Lie algebroid (Example \ref{PoissonCotangent}).  Integrable Poisson Lie algebroids are associated to symplectic groupoids \cite{Wei87,Cat01}. 
\end{example}
\begin{example}[Universal]\label{CourantLie}
	The `universal' action of Kotov--Strobl \cite{Cha17} is given by
	\begin{align*}
	S_{\text{univ}}[X,V,W]=\int_{\Sigma}\Big(\frac{1}{2}g_{\mu\nu}DX^\mu\wedge\star DX^\nu+W_\mu\wedge(\d X^\mu-\frac{1}{2}V^\mu)\Big)+\int_{\Sigma_3}H,
	\end{align*}
	where $DX^\mu=\d X^\mu-V^\mu$, $V\in\Omega^1(\Sigma,X^{**}T\Sigma)$ and $W\in\Omega^1(\Sigma,X^{**}T^*\Sigma)$.  This is equivalent to $S_Q[X,A]$ through the identification $V^\mu=\check{\rho}^\mu_aA^a$ and $ W_\mu=\check{C}_{\mu\nu}\check{\rho}^\nu_aA^a$.
	
	The Kotov--Strobl action and Poisson sigma model are special cases of the universal action. The Kotov--Strobl action is found by identifying
	$V^\mu=\rho^\mu_aA^a$, and $W_\mu=\alpha_{\mu a} A^a$. The Poisson sigma model (for $Q=T^*M$) is identified via $V^\mu=\pi^{\nu\mu}A_\nu$, and $W_\mu=A_\mu$. The universal action naturally interpolates between the Poisson sigma model and the Kotov--Strobl model.  The Lie algebroid is found by restricting the Lie bialgebroid constructed from $[\cdot,\cdot]$, $[\cdot,\cdot]_\pi$ to a (small) Dirac\footnote{A  small Dirac structure is an involutive and isotropic subbundle.  A Dirac structure is a small Dirac structure  of maximal dimension.} structure (see \cite{Kos07} for a description of Lie bialgebroids).
	
	The universal title refers to the fact each different choice of small Dirac structure on $TM\oplus T^*M$ gives a different model. 
\end{example}
\begin{example}[Foliation]\label{FoliationEx}
	Let $Q\subset TM$ be an involutive subbundle (a constant rank subbundle closed under the Lie bracket).  This defines a Lie algebroid (Example \ref{FoliatedLiealgebroid}).  
	
	An involutive linearly independent set of vector fields $v_a$ will satisfy (by definition)
	\begin{align*}
	[v_a,v_b]=C^c{}_{ab}v_c,\quad C^c{}_{ab}\in C^\infty(M),
	\end{align*}
	but in general $C^c{}_{ab}$ will not be constant.  Involutive foliated vector fields $Q=T\cF\subset TM$ are not generated by group actions in general. The integrability of smooth distributions is given by the Steffan--Sussmann conditions (see \cite{Ste73}). 
\end{example}

\begin{example}[Lie Groupoid]\label{GroupoidExample}
	Take any Lie groupoid $(\cG,M)$, and take the Lie algebroid $Q=\Lie(\cG)$.   The action $S_Q[X,A]$ provides a sigma model action. 
\end{example}

It is possible to consider sigma models for all the examples of Lie algebroids considered in Section \ref{Sec:LieAlgebroids} (Examples \eqref{LieAlgebraEx}--\eqref{PoissonCotangent}).  Integrability is guaranteed when you start with a Lie groupoid.  Examples of Lie groupoids can be found in Section \ref{Sec:LieAlgGroup} (Examples \eqref{LieGroupEx}--\eqref{SymplecticGroupoid}).

\begin{remark}
	We can consider topologically non-trivial examples by considering the gauge groupoid $\cG(P)$ (Example \ref{GaugeGroupoid}) corresponding to a topologically non-trivial principal bundle $P(M,\pi,\sfG)$.	
\end{remark}

It is clear that many natural examples of non-linear sigma models can be described using the action $S_Q[X,A]$.  The task remaining is to understand when such models can be gauged.


\section{Closure of gauge algebroid}\label{Sec:GaugeAlgebroid}
Closure of the gauge algebroid on the fields $X$ and $A$ is an important constraint on the Lie algebroid gauging procedure.  The gauged model will be equivalent to the original model if there exists a finite gauge transformation which sets the gauge field $A$ to zero. A necessary condition for the infinitesimal gauge transformations \eqref{GaugeAlg+} to integrate to finite gauge transformations is closure of the gauge algebroid.   Closure of the gauge algebroid does not follow automatically from \eqref{GaugeAlg+} and provides a real constraint to the application of Lie algebroid gauging.  We note that the formulae in this section have been calculated using the Lie algebroid $([\cdot,\cdot]_{X^{**}Q},\check{\rho})$; we have left the pullback maps implied throughout in order to better connect to the literature.

Closure of the gauge algebroid for the variation of the field $X$ gives the anchor homomorphism property of a Lie algebroid:
\begin{align}
[\delta_{\varepsilon_1},\delta_{\varepsilon_2}]X=\delta_{[\varepsilon_1,\varepsilon_2]_Q}X\quad \Longleftrightarrow \quad [\rho(\varepsilon_1),\rho(\varepsilon_2)]X=\rho([\varepsilon_1,\varepsilon_2]_Q)X.
\end{align}
Closure on the gauge field $A$ is more involved.  It is convenient to define $D_\pm X:=\frac{1}{2}(DX\pm\star DX)$.  We note that $D_\pm X^\mu$ transforms nicely 
\begin{align*}
\delta_\varepsilon D_\pm X^\mu=&\varepsilon^a(\partial_\nu\rho^\mu_a-\rho^\mu_b(\Omega^\pm)^b{}_{\nu a})D_\pm X^\nu.
\end{align*}
A straightforward but lengthy calculation gives: 
\begin{align}\label{Eq:ClosureComponents}
([\delta_{\varepsilon_1},\delta_{\varepsilon_2}]-\delta_{[\varepsilon_1,\varepsilon_2]_Q})A^a
=&\varepsilon^b_1\varepsilon^c_2(-\nabla^-_\mu (T_{\nabla^-})^a{}_{bc}+2\rho^\nu_{[b|}(R_{\nabla^-})^a{}_{\nu \mu| c]})D_-X^\mu\\
+&\varepsilon^b_1\varepsilon^c_2(-\nabla^+_\mu (T_{\nabla^+})^a{}_{bc}+2\rho^\nu_{[b|}(R_{\nabla^+})^a{}_{\nu \mu |c]})D_+X^\mu,\nonumber
\end{align}
where 
\begin{align*}
(T_{\nabla^\pm})^a{}_{bc}=&\rho^\mu_b(\Omega^\pm)^a{}_{\mu c}-\rho^\mu_c(\Omega^\pm)^a{}_{\mu b}-C^a{}_{bc};\\
(R_{\nabla^\pm})^a{}_{\mu\nu b}=&2\partial_{[\mu}(\Omega^\pm)^a{}_{\nu]b}+2(\Omega^\pm)^a{}_{[\mu| c}(\Omega^\pm)^c{}_{|\nu] b};\\
(\nabla^\pm_{\mu}T_{\nabla})^a{}_{bc}=&\partial_\mu (T_{\nabla})^a{}_{bc}+(T_\nabla)^d{}_{bc}(\Omega^\pm)^a{}_{\mu d}-(T_\nabla)^a{}_{dc}(\Omega^\pm)^d{}_{\mu b}-(T_\nabla)^a{}_{bd}(\Omega^\pm)^d{}_{\mu c}.
\end{align*}
The closure constraints for the special case $\Omega^+=\Omega^-$ have appeared in the literature before \cite{May08}.\footnote{We replace the Lie bracket $[\varepsilon_1,\varepsilon_2]_{X^*Q}$ with the Lie algebroid bracket $[\varepsilon_1,\varepsilon_2]_{X^{**}Q}$ but the expression is formally the same.}

Closure of the gauge algebroid requires that either $D_\pm X\equiv 0$ or $\nabla^\pm_\mu (T_{\nabla^\pm})^a{}_{bc}=2\rho^\nu_{[b|}(R_{\nabla^\pm})^a{}_{\nu \mu |c]}$.  If $D_\pm X\equiv 0$ the Lagrangian is zero.  While it may be possible to consider $D_\pm X=0$ as an `on-shell' condition it is not clear that this is a natural constraint physically. In the standard case of Lie algebra gauging the gauge algebra holds identically without assuming $D_\pm X^\mu=0$.  In addition, quantisation requires `off-shell' closure of the gauge algebroid.  It makes sense to require closure for all $X:\Sigma\rightarrow M$, including $D_\pm X\neq 0$.  With this assumption it follows that
\begin{equation}\label{Qflat1}
([\delta_{\varepsilon_1},\delta_{\varepsilon_2}]-\delta_{[\varepsilon_1,\varepsilon_2]_Q})A^a=0\quad \Longleftrightarrow \quad \nabla^\pm_\mu (T_{\nabla^\pm})^a{}_{bc }=2\rho^\nu_{[b|}(R_{\nabla^\pm})^a{}_{\nu\mu|c]}.
\end{equation}

Closure of the gauge algebroid places a curious constraint on the connections $\nabla^\pm$---a constraint that is not readily interpreted.  When $\nabla^\pm$ are flat connections, the closure of the gauge algebra acquires a nice interpretation. In this case  $\nabla^\pm (T_{\nabla^\pm})^a{}_{bc}=0$, which is equivalent to the fact that covariantly constant sections are closed on $[\cdot\ ,\cdot ]_Q$ i.e.
\begin{align*}
(\nabla^\pm q_1=0\ \&\ \nabla^\pm q_2=0\quad \Rightarrow\quad \nabla^\pm [q_1,q_2]_Q=0)\quad \Longleftrightarrow\quad  \nabla^\pm_\mu (T_{\nabla^\pm})^a{}_{bc}=0,
\end{align*}
for all $q_1,q_2\in\Gamma(Q)$.  The constraint for $R_{\nabla^\pm}\neq 0$ is best understood by `lifting' the $TM$-connections on $Q$ to $Q$-connections on $Q$:  
\begin{align}\label{Qconnection}
{}^Q\nabla^\pm_{q_1}q_2:=\nabla^\pm_{\rho(q_2)}q_1+[q_1,q_2]_Q.
\end{align}	
The adjoint connection, Equation \eqref{Qconnection}, is well known in the mathematics literature and plays an important role in integrating $Q$-paths in a Lie groupoid \cite{Cra05}.  In a local frame for $Q$, specified by a basis $\{ e_a \}$, the components of the adjoint connections are given by
\begin{align}\label{Qcomponents}
({}^Q\Omega^\pm)^a{}_{bc}=\rho^\mu_c(\Omega^\pm)^a{}_{\mu b}+C^a{}_{bc}.
\end{align}	

The torsion and curvature of the $Q$-connections ${}^Q\nabla^\pm$ can be related to the torsion and curvature of the $Q$-connections $\overline{\nabla}^\pm_{q_1}q_2:=\nabla^\pm_{\rho(q_1)}q_2$, yielding the following: 
\begin{align*}
T_{{}^Q\nabla^\pm}(q_1,q_2)=&-T_{\overline{\nabla}^\pm}(q_1,q_2),\\	
R_{{}^Q\nabla^\pm}(q_1,q_2)(q_3)=&R_{\overline{\nabla}^\pm}(q_1,q_3)q_2-R_{\overline{\nabla}^\pm}(q_2,q_3)q_1\\
&+\overline{\nabla}^\pm_{q_3} T_{\nabla^\pm}(q_1,q_2)-T_{\overline{\nabla}^\pm}(\overline{\nabla}^\pm_{q_3} q_1,q_2)-T_{\overline{\nabla}^\pm}(q_1,\overline{\nabla}^\pm_{q_3} q_2).
\end{align*}

Calculating $R_{{}^Q\nabla^\pm}$ in local coordinates produces the algebroid closure constraint \eqref{Qflat1}.  In summary: 
\begin{align}\label{Eq:FlatnessConstraint}
([\delta_{\varepsilon_1},\delta_{\varepsilon_2}]-\delta_{[\varepsilon_1,\varepsilon_2]_Q})A=0\quad \Longrightarrow \quad R_{{}^Q\nabla^\pm}=0.	
\end{align}

In \cite{Kot15} it has been noted that for the special case $\Omega^+=\Omega^-$ the closure of the gauge algebroid implies the $Q$-flat condition.  

\begin{remark}
If $\rho$ has a non-trivial kernel then $R_{{}^Q\nabla^\pm}=0$ does not necessarily imply closure of the gauge algebroid. If $\{\rho_a \}$ are linearly independent we  can define $Q\cong T\cF\subset TM$ to be the span of $\{\rho_a \}$. In this case $\rho$ has no kernel and $R_{{}^Q\nabla^\pm}=0$ implies the closure of the gauge algebroid. 
\end{remark}

It follows from Theorem \ref{GroupoidMorphismThrm} that ${}^{X^{**}Q}\nabla^\pm$ will define representations of Lie algebroids corresponding to Lie groupoids $\cG^\pm(X^{**}Q)$ if ${}^{Q}\nabla^\pm$ define a representation of  $\cG^\pm(Q)$.  Thus it is sufficient to study the flatness of the connections ${}^Q\nabla^\pm$.  Alternatively we come to the same conclusion by noting that $R_{{}^Q\nabla^\pm}$ are tensors, so  $R_{X^{**}{}^Q\nabla^\pm}=X^{**}R_{{}^Q\nabla^\pm}$, and it is sufficient to check the flatness of ${}^Q\nabla^\pm$. 

Closure of the gauge algebra implies the existence of two flat $Q$-connections on $Q$.  These connections define representations of Lie algebroids on $Q$. The Lie algebroids define paths associated to the Weinstein groupoid (as outlined in Section \ref{Sec:LieAlgGroup}).  The paths define the orbits describing our gauge symmetry.  

If ${}^Q\nabla^\pm$ are not flat then we do not have a representation of a Lie algebra.  The action of our algebroid is not even locally integrable. It seems that the notion of gauging is not well defined in this case.

We close this section with the observation that the flatness condition $R_{{}^Q\nabla^\pm}=0$ implies that ${}^Q\Omega^\pm$ are Maurer--Cartan forms for the frame bundle $\cB(Q)$ over $M$.  It follows that 
\begin{align}\label{FormofQflat}
{}^Q\Omega^\pm=K_\pm^{-1} \d{}_QK_\pm,\quad ({}^Q\Omega^\pm)^a{}_{ bc}=(K^{-1}_\pm)^a{}_{d}\rho^\mu_b\partial_\mu(K_\pm)^d{}_{c},
\end{align}	
for some $K_\pm\in C^\infty(M,\sfGL(d))$ (where $d=\dim(Q)$) and $\d{}_Q:C^\infty(M)\rightarrow \Omega^1(M)$ is given by $e^a\rho^\mu_a\partial_\mu$, for some local basis $\{e^a \}$ for $Q^*$.

\section{Solving the generalised Killing equation}\label{Sec:GenSolution}

In Section \ref{Sec:GaugeAlgebroid} it was shown that closure of the gauge algebroid on the gauge field $A$ implies that $R_{{}^Q\nabla^\pm}=0$.  This constraint allows an explicit construction of solutions to the gauging constraints \eqref{Ecompatibility} and \eqref{Eq:DefineC}. This section derives the necessary and sufficient conditions on an involutive set of vector fields in order to satisfy the generalised Killing equation and result in a closed gauge algebroid. These conditions will be important for determining the extent to which Lie algebroid gauging is applicable to non-linear sigma models. 

First we will consider the case of gauging with respect to a set of linearly independent vector fields $\{\rho_a \}$, $a=1,\dots,k$ where $k\leq \dim(M)$.  In such a case $\rank(\text{Im}(\rho))=k$ and corresponds to gauging with respect to a regular Lie algebroid.  The general case, allowing $\rank(\text{Im}(\rho))$ to vary at different points in $M$, is treated separately in Section \ref{Sec:GaugingGenVecFields}.  Section \ref{Sec:Examples} gives examples of Lie algebroid gauging, demonstrating the results of this chapter.

We note that Equation \eqref{Eq:DefineC} can be used to determine a choice of field $C\in\Omega^2(M)$ and does not pose a constraint to gauging.  Our first task is to find the general solution to the generalised Killing equation \eqref{Ecompatibility}.  Written in matrix form the generalised Killing condition is given by a system of matrix equations
\begin{align}
\cL_{\rho_a}E=&E\rho\Omega^+_a+E^T\rho\Omega^-_a.
\end{align}
For each $a=1,\dots,k$ this is a linear matrix equation and a solution for $\Omega^+_a$ exists if and only if
\begin{align}\label{InvertibleGen}
\cL_{\rho_a}E-E^T\rho\Omega^-_a=E\rho(E\rho)^+(\cL_{\rho_a}E-E^T\rho\Omega^-_a),
\end{align}
for each $a$, where $(E\rho)^+$ denotes the Moore--Penrose pseudo-inverse of $E\rho$.  If this constraint is satisfied the solution $\Omega^+_a$ is given by
\begin{align}
\Omega^+_a=&(E\rho)^+(\cL_{\rho_a}E-E^T\rho\Omega^-_a)+(I-(E\rho)^+E\rho)X_a,\label{eq:GenSolutionOmega+}
\end{align}
where $I$ is the $k\times k$ identity matrix, and  $X_a\in C^\infty(M,\RR^{k\times n})$ is arbitrary.\footnote{Remember that $a,b,c=1,\dots,k$ and $\mu,\nu,\lambda=1,\dots,n=\dim(M)$.}  In explicit index notation we have
\begin{align*}
(\Omega^+_a)^b{}_{\mu}=&((E\rho)^+)^b{}_{\lambda}(\cL_{\rho_a}E)^\lambda{}_{\mu}-((E\rho)^+)^b{}_{\lambda}E_{\kappa\lambda}\rho^\kappa_c(\Omega^-)^c{}_{\mu a}\\
&+(\delta^b{}_{c}-((E\rho)^+)^b{}_{\lambda}(E\rho)^\lambda_{c})(X_a)^c{}_{\mu}.
\end{align*}

\begin{theorem}\label{Th:GaugeThrm}
	Given a field $E\in\Gamma(T^*M\otimes T^*M)$ and a choice of involutive vector fields $\rho_a\in\Gamma(TM)$ ($a=1,\dots,k$) defining a Lie algebroid $[\rho_a,\rho_b]=C^c{}_{ab}\rho_c$, the generalised Killing equation
	\begin{align*}
	(\cL_{\rho_a}E)_{\mu\nu}=&E_{\mu\lambda}\rho^\lambda_b(\Omega^+)^b{}_{\nu a}+E_{\lambda\nu}\rho^\lambda_b(\Omega^-)^b{}_{\mu a},
	\end{align*}
	has solutions for $\Omega^\pm\in\Gamma(T^*M\otimes Q)$ if and only if 
	\begin{align}
	\cL_{\rho_a}E-E^T\rho\Omega^-_a=E\rho(E\rho)^+(\cL_{\rho_a}E-E^T\rho\Omega^-_a)
	\end{align}
	holds for some set $\{\Omega^-_a\}$.  If there exists a set $\{\rho_a,\Omega^-_a\}$, satisfying the stated conditions, then solutions are given by 
	\begin{align}\label{Eq:Necessary}
	\Omega^+_a=&(E\rho)^+(\cL_{\rho_a}E-E^T\rho\Omega^-_a)+(I-(E\rho)^+E\rho)X_a,
	\end{align}
	where $X_a\in C^\infty(M,\RR^{k\times n})$ is arbitrary.  In addition, any solution $\Omega^+_a$ satisfies \eqref{Eq:Necessary}. 
\end{theorem}

It appears that the general conditions for solving the generalised Killing equation have not appeared in the literature before.  In \cite{Cha16b} a particular solution for $\phi$ and $\omega$ was found in the special case that $(\alpha^*+\rho)^{-1}$ exists.\footnote{We refer the reader to the paper \cite{Cha16b} for notation and details.}  It follows from Theorem \ref{Th:GaugeThrm} that a family of solutions exist in this case, with the given choice being a particular solution.  No statement was made about the closure of the gauge algebroid for this choice, nor the existence of solutions when $(\alpha^*+\rho)^{-1}$ does not exist.  

Having found the general solution to the generalised Killing equation \eqref{Ecompatibility} we now turn our attention to those which result in a closed gauge algebroid.  These are gaugings which are at least locally integrable.  We will find $\Omega^\pm_a$ which satisfy the generalised Killing equation and give ${}^Q\Omega^\pm_a$ which define flat connections ${}^Q\nabla^\pm$.  ~\\

\subsection{Gauging with linearly independent vector fields}~\\
In this section solutions to the gauging constraints (Equations \eqref{Ecompatibility} and \eqref{Eq:FlatnessConstraint}) for an arbitrary choice of linearly independent involutive vector fields $\{\rho_a\}$ where $a=1,\dots,k$ and $k\leq \dim(M)$.  This allows us to describe gauging with respect to regular Lie algebroids $[\rho_a,\rho_b]=C^c{}_{ab}\rho_c$.  The components $\rho^\mu{}_a$ define a $n\times k$ matrix. As a consequence of the linear independence of $\{\rho_a\}$, the columns of $\rho$ are linearly independent and $\rho^T\rho$ is invertible.  In this case we have an explicit formula for the pseudo-inverse  
\begin{align*}
\rho^+=(\rho^T\rho)^{-1}\rho^T.
\end{align*}
It follows that $\rho^+$ is a \emph{left-inverse} for $\rho$ i.e. $\rho^+\rho=I$ the $k\times k$ identity matrix.

Throughout this chapter we assume that $E$ is invertible (this is true for almost all physical models) and hence satisfies $\rank(E)=n$.  As $E$ defines an endomorphism it follows that $\rank(E\rho)=\rank(\rho)$. $E\rho$ is a $n\times k$ matrix, and has linearly independent columns 
\begin{align*}
(E\rho)^+=(\rho^TE^TE\rho)^{-1}\rho^TE^T
\end{align*}
and $(E\rho)^+$ is a \emph{left-inverse} for $E\rho$
\begin{align}
(E\rho)^+E\rho=I.
\end{align}

Assuming that $\{\rho_a \}$ are linearly independent simplifies the general solution to the generalised Killing equation \eqref{Ecompatibility}: 
\begin{align}\label{eq:GenKillingEqnSoln}
\Omega^+_a=&(E\rho)^+(\cL_{\rho_a}E-E^T\rho\Omega^-_a).
\end{align}
When condition \eqref{InvertibleGen} holds, the generalised Killing equation has solutions, and they are given by \eqref{eq:GenKillingEqnSoln}.

In order to impose the $Q$-flatness condition on ${}^Q\nabla^\pm$ it is useful to get an expression for the components $({}^Q\Omega^+_a)^b{}_{c}$. Using \eqref{Qcomponents} we have:
\begin{align*}
({}^Q\Omega^+_a)^b{}_{c}=&(\Omega^+_a)^b{}_{\mu}\rho^\mu_c+C^b{}_{ac}\\
=&(E\rho)^+(\cL_{\rho_a}E-E^T\rho\Omega^-_a)^b{}_{\mu}\rho^\mu_c+C^b{}_{ac}\\
=&((E\rho)^+)^b{}_{\lambda}(\cL_{\rho_a}E)_{\lambda\mu}\rho^\mu_c-((E\rho)^+)^b{}_{\lambda}E_{\kappa\lambda}\rho^\kappa_d(\Omega^-_a)^d{}_{\mu}\rho^\mu_c+C^b{}_{ac}\\
=&((E\rho)^+)^b{}_{\lambda}(\cL_{\rho_a}E)_{\lambda\mu}\rho^\mu_c-((E\rho)^+)^b{}_{\lambda}E_{\kappa\lambda}\rho^\kappa_d(({}^Q\Omega^-)^d{}_{ac}-C^d{}_{ac})+C^b{}_{ac}.
\end{align*}
This expression can be simplified: 
\begin{lemma} The following identity holds:
	\begin{align}\label{eq:Lemma1}
	(\cL_{\rho_a}E)_{\lambda\mu}\rho^\mu_c=(\rho^+)^d_\lambda\rho^\kappa_a\partial_\kappa(\rho^\nu_dE_{\nu\mu}\rho^\mu_c)-(\rho^+)^e_\lambda C^d{}_{ae}\rho^\kappa_dE_{\kappa\mu}\rho^\mu_c-E_{\lambda\mu}C^d{}_{ac}\rho^\mu_d.
	\end{align}
\end{lemma}
\begin{proof}
	Recall that $[\rho_a,\rho_b]:=C^c{}_{ab}\rho_c$. By direct computation
	\begin{align}
	(\cL_{\rho_a}E)_{\lambda\mu}\rho^\mu_c=&\rho^\kappa_a(\partial_\kappa E_{\lambda\mu})\rho^\mu_c+(\partial_\lambda \rho^\kappa_a)E_{\kappa\mu}\rho^\mu_c+(\partial_\mu \rho^\kappa_a)E_{\lambda\kappa}\rho^\mu_c\nonumber\\
	=&\rho_a^\kappa(\partial_\kappa E_{\lambda\mu}\rho^\mu_c)-E_{\lambda\mu}(\rho^\kappa_a\partial_\kappa\rho^\mu_c-\rho^\kappa_c\partial_\kappa\rho^\mu_a)+(\partial_\lambda\rho^\kappa_a)E_{\kappa\mu}\rho^\mu_c\nonumber\\
	=&\rho_a^\kappa(\partial_\kappa E_{\lambda\mu}\rho^\mu_c)-E_{\lambda\mu}([\rho_a,\rho_c])^\mu+(\partial_\lambda\rho^\kappa_a)E_{\kappa\mu}\rho^\mu_c\nonumber\\
	=&\rho_a^\kappa(\partial_\kappa E_{\lambda\mu}\rho^\mu_c)-E_{\lambda\mu}C^d{}_{ac}\rho^\mu_d+(\partial_\lambda\rho^\kappa_a)E_{\kappa\mu}\rho^\mu_c.\label{eq:LemmaId1}
	\end{align}
	Next we use the fact that $\rho^+\rho=I$, 
	\begin{align}\label{eq:LemmaId2}
	\rho^\kappa_a\partial_\kappa (E_{\lambda \mu}\rho^\mu_c)=&\rho^\kappa_a\partial_\kappa((\rho^+)^d_\lambda\rho^\nu_d E_{\nu \mu}\rho^\mu_c)=\rho^\kappa_a\partial_\kappa((\rho^+)^d_\lambda)\rho^\nu_d E_{\nu \mu}\rho^\mu_c+\rho^\kappa_a(\rho^+)^d_\lambda\partial_\kappa(\rho^\nu_d E_{\nu \mu}\rho^\mu_c).
	\end{align} 
	Using $d(\rho^+\rho)=dI=0$ to conclude that $(d\rho^+)\rho=-\rho^+d\rho$, we have
	\begin{align}
	\rho^\kappa_a\partial_\kappa((\rho^+)^d_\lambda)\rho^\nu_d E_{\nu \mu}\rho^\mu_c=&-(\rho^+)^d_\lambda\rho^\kappa_a\partial_\kappa\rho^\nu_d E_{\nu \mu}\rho^\mu_c\nonumber\\
	=&(\rho^+)^d_\lambda C^e{}_{da}\rho^\nu_e E_{\nu\mu}\rho^\mu_c-(\rho^+)^d_\lambda\rho^\kappa_d\partial_\kappa\rho^\nu_a E_{\nu\mu}\rho^\mu_c\nonumber\\
	=&(\rho^+)^e_\lambda C^d{}_{ea}\rho^\kappa_d E_{\kappa\mu}\rho^\mu_c-(\partial_\lambda\rho^\nu_a) E_{\nu\mu}\rho^\mu_c.\label{eq:LemmaId3}
	\end{align}
	Combining \eqref{eq:LemmaId1}, \eqref{eq:LemmaId2} and \eqref{eq:LemmaId3} gives the desired result \eqref{eq:Lemma1}.
\end{proof}
Substituting \eqref{eq:Lemma1} in the expression for ${}^Q\Omega^+_a$ gives
\begin{align*}
({}^Q\Omega^+_a)^b{}_{c}=&((E\rho)^+)^b{}_{\lambda}(\rho^+)^d_\lambda\rho^\kappa_a\partial_\kappa( \rho^\nu_dE_{\nu\mu}\rho^\mu_c)+(\delta^b{}_d-((E\rho)^+E\rho)^b{}_d)C^d{}_{ac}\\
&-((E\rho)^+)^b{}_{\lambda}E_{\kappa\lambda}\rho^\kappa_d({}^Q\Omega^-)^d{}_{ac}\\
=&((\rho^T E\rho)^+)^b{}_{\lambda}\rho^\kappa_a\partial_\kappa(\rho^T E\rho)^\lambda{}_c-(E\rho)^+E^T\rho({}^Q\Omega^-_a)^{b}{}_{c}.
\end{align*}
This can be written more succinctly as
\begin{align}\label{eq:QOmega+form}
{}^Q\Omega^+_a=(\rho^T E\rho)^+\rho_a(\rho^T E\rho)-(E\rho)^+E^T\rho{}^Q\Omega^-_a.
\end{align}
Whenever \eqref{InvertibleGen} holds there exist solutions to the gauging constraint \eqref{Ecompatibility} with the general solution given by \eqref{eq:QOmega+form}.  

We wish to make a choice of connection coefficients ${}^Q\Omega^-$ such that ${}^Q\nabla^\pm$ are flat.  If ${}^Q\nabla^-$ is flat there exists a choice of frame $\tilde{e}_b$ such that $({}^Q\widetilde{\Omega}^-)^a{}_{bc}=0$.  Take $N\in \End(Q)$, such that $\tilde{e}_b=N^a{}_{b}e_a$.  Now $\tilde{\rho}_b=N^a{}_{b}\rho_a$ and in this frame 
\begin{align}
{}^Q\widetilde{\Omega}^+_a=(\tilde{\rho}^TE\tilde{\rho})^+\tilde{\rho}_a(\tilde{\rho}^TE\tilde{\rho}).
\end{align} 
Recall from Equation \eqref{FormofQflat} that the coefficients ${}^Q\Omega^\pm_a$ define flat connections ${}^Q\nabla^\pm$ if and only if they are of the form 
\begin{align}\label{eq:QFlatForm}
{}^Q\Omega^\pm_a=(K_\pm)^{-1}\rho_a(K_\pm)=(K_\pm)^{-1}\rho^\mu_a\partial_\mu K_\pm,
\end{align}
for some $K_\pm\in C^\infty(M,\sfGL(k))$.  We require that 
\begin{align*}
{}^Q\widetilde{\Omega}^+_a=(\widetilde{K}_+)^{-1}\tilde{\rho}_a(\widetilde{K}_+)=(\tilde{\rho}^TE\tilde{\rho})^+\tilde{\rho}_a(\tilde{\rho}^TE\tilde{\rho}),
\end{align*}
which is satisfied if $\widetilde{K}_+=\tilde{\rho}^TE\tilde{\rho}$ and $\tilde{\rho}^TE\tilde{\rho}$ is invertible.  This provides another necessary condition for ${}^Q\widetilde{\Omega}^+_a$ to be flat.  Furthermore we note that
\begin{align*}
\rank(\tilde{\rho}^TE\tilde{\rho})=\rank(\rho^TE\rho),
\end{align*}
as $\tilde{\rho}_a$ and $\rho_{b}$ are related by $N\in\End(Q)$.  If ${}^Q\nabla^+$ is flat then $(\rho^TE\rho)$ is invertible, and $(\rho^T E\rho)^+=(\rho^T E\rho)^{-1}$. Returning to \eqref{eq:QOmega+form} we find    
\begin{align}
{}^Q\Omega^+_a=(\rho^T E\rho)^{-1}\rho_a(\rho^T E\rho)-(E\rho)^+E^T\rho{}^Q\Omega^-_a,
\end{align}
and have a solution for ${}^Q\Omega^-_a=0$, ${}^Q\Omega^+=(K_+)^{-1}\rho^\mu_a\partial_\mu K_+$ with $K_+=\rho^TE\rho$.  In local coordinates
\begin{align*}
({}^Q\Omega^-)^a{}_{bc}=0=(\Omega^-)^a{}_{\mu b}\rho^\mu_c+C^a{}_{bc},\quad \Longrightarrow\quad (\Omega^-)^a{}_{\mu b}=-(\rho^+)^c_\mu C^a{}_{bc}.
\end{align*}
Let us define $\Psi_a\in\Gamma(T^*M\otimes T^*M)$ as
\begin{align}\label{eq:PsiGaugeableCond}
(\Psi_a)_{\mu\nu}:=(\cL_{\rho_a}E)_{\mu\nu}-E_{\lambda\nu}\rho^\lambda_bC^b{}_{ac}(\rho^+)^c_\mu.
\end{align}
Taking ${}^Q\Omega^-_a=0$ modifies Equation \eqref{eq:GenSolutionOmega+} (the necessary condition for the existence of $\Omega^+_a$) so that consistency requires
\begin{align}\label{eq:NeccesaryConstistency}
\Psi_a=E\rho(E\rho)^+\Psi_a.
\end{align}

We summarise the preceding calculations in the form of a theorem for gauging with respect to a choice of linearly independent involutive vector fields:
\begin{theorem}\label{Thrm:GaugeEverything}
	Given an invertible field $E\in\Gamma(T^*M\otimes T^*M)$ and a choice of involutive linearly independent vector fields $\rho_a\in\Gamma(TM)$, $a=1,\dots,k$, defining a regular Lie algebroid $[\rho_a,\rho_b]=C^c{}_{ab}\rho_c$, the generalised Killing equation
	\begin{align*}
	(\cL_{\rho_a}E)_{\mu\nu}=&E_{\mu\lambda}\rho^\lambda_b(\Omega^+)^b{}_{\nu a}+E_{\lambda\nu}\rho^\lambda_b(\Omega^-)^b{}_{\mu a},
	\end{align*}
	has solutions for $\Omega^\pm_a\in\Gamma(T^*M\otimes Q)$ if and only if 
	\begin{align}\label{eq:GaugingConditions}
	\Psi_a=E\rho(E\rho)^+\Psi_a,\quad \det(\rho^TE\rho)\neq 0,
	\end{align}
	where 
	\begin{align}
	(\Psi_a)_{\mu\nu}:=(\cL_{\rho_a}E)_{\mu\nu}-E_{\lambda\nu}\rho^\lambda_bC^b{}_{ac}(\rho^+)^c_\mu,\quad (E\rho)^+=(\rho^TE^TE\rho)^{-1}\rho^TE^T.
	\end{align}
	If these conditions are satisfied then a solution is given by 
	\begin{align}\label{eq:Omegasolns}
	(\Omega^-)^a{}_{\mu b}=-(\rho^+)^c_\mu C^a{}_{bc},\quad\quad (\Omega^+)^a{}_{\mu b}=((E\rho)^+)^a{}_{\lambda}(\Psi_b)_{\lambda\mu},
	\end{align}
	where $(\Omega^\pm)^a{}_{\mu b}$ are connection coefficients defining $\nabla^\pm$.  The corresponding flat adjoint connections ${}^Q\nabla^\pm$ are defined by
	\begin{align}
	{}^Q\Omega^-_a=0,\quad\quad {}^Q\Omega^+_a=(\rho^T E\rho)^{-1}\rho^\lambda_a\partial_\lambda(\rho^T E\rho). 
	\end{align}
\end{theorem}	

Theorem \ref{Thrm:GaugeEverything} gives one solution for gauging (when the necessary condition \eqref{eq:GaugingConditions} is satisfied).  This choice is certainly not unique.  In fact, if $\rho$ is invertible then for any choice of flat connection ${}^Q\nabla^-$ it is possible to find a flat ${}^Q\nabla^+$ which satisfies the generalised Killing equation (Corollary \ref{Thrm:FlatConnectionChoice}).
~\\
\subsection{Gauging general vector fields }\label{Sec:GaugingGenVecFields}~\\
In this section we consider Lie algebroid gauging for a general set of involutive vector fields $\{\rho_a \}$---dropping the linear independence requirement.  A different perspective is taken from the last section.  If the generalised Killing equation holds, then a projected version of the equation holds   \eqref{Eq:ProjGenKilling}.  This projected form can be `lifted' to an equation on the vector bundle $Q$ described using Lie algebroid geometry. An expression defining the flat adjoint connections ${}^Q\nabla^\pm$ can be found in terms of objects on the vector bundle $Q$.   

A general set of involutive vector fields may become linearly dependent for some set of points in $M$. We consider a fixed number of vector fields $k\leq \dim(M)$, but the image of the distribution spanned by the vector fields is allowed to change dimension.  A set of involutive vector fields describes a generalised distribution and defines a singular foliation.  Not all singular foliations are generated by vector fields in this way and we restrict ourselves to those which do. As a concrete example of the type of vector fields we are interested in considering the following: 
\begin{example}[Generalised distribution]
	Let $M=\RR^2$ with coordinates $\{x,y \}$.  Consider the vector fields
	\begin{align*}
	X=\frac{\partial}{\partial x},\quad Y=\frac{\partial}{\partial x}+f(x)\frac{\partial}{\partial y},
	\end{align*}
	where $f(x)\in C^\infty(M)$ satisfies $f(x)=0$ for $x\leq 0$, and $f(x)>0$ for $x>0$. Let $\Delta$ be the distribution given by the span of $X$ and $Y$.  At each point $x\in M$, $\Delta_x\subseteq TM_x$ is the vector space given by the linear span of $X(x)$ and $Y(x)$.	For $x>0$ $\dim(\Delta)=2$, and for $x\leq 0$ $\dim(\Delta)=1$. 
	\begin{align*}
	[X,Y]=f'(x)\partial_y.
	\end{align*}
	At each point $x_0\in M$ we have that $[X,Y](x_0)$ is in the span of $\{X(x_0),Y(x_0)\}$ and the distribution is involutive. To be explicit, when $x>0$ $[X,Y](x_0)=\ln(f)'Y(x_0)-\ln(f)'X(x_0)$ and when $x\leq 0$ $[X,Y]=0$. By the Stefan--Sussman theorem the distribution $\Delta$ describes a singular foliation, with $\Delta$ giving the partition into leaves.  For a detailed discussion of this example see \cite{Ste73}.
\end{example}

Theorem \ref{Th:GaugeThrm} does not assume that $\{\rho_a \}$ are linearly independent and still applies as a necessary condition on solving the generalised Killing equation \eqref{Ecompatibility} in the general case.  The issue remaining is the possible construction of flat connections ${}^Q\nabla^\pm$ to ensure closure of the gauge algebroid.  If a set of vector fields $\{\rho_a\}$ satisfy the generalised Killing equation      
\begin{align*}
(\cL_{\rho_a}E)_{\mu\nu}=E_{\mu\lambda}\rho^\lambda_b(\Omega^+)^b{}_{\nu a}+E_{\lambda\nu}\rho^\lambda_b(\Omega^-)^b{}_{\mu a},
\end{align*}
it follows that 
\begin{align}\label{Eq:ProjGenKilling}
\rho^\mu_d\rho^\nu_c(\cL_{\rho_a}E)_{\mu\nu}=\rho^\mu_d\rho^\nu_cE_{\mu\lambda}\rho^\lambda_b(\Omega^+)^b{}_{\nu a}+\rho^\mu_d\rho^\nu_cE_{\lambda\nu}\rho^\lambda_b(\Omega^-)^b{}_{\mu a}.
\end{align}
The converse is of course not true in general; Equation \eqref{Eq:ProjGenKilling} does not imply that the generalised Killing equation is satisfied. It is instructive to `lift' Equation \eqref{Eq:ProjGenKilling} to the vector bundle $Q$.  By `lift' we mean rewriting $\Omega^\pm$ (which define the $TM$-connections on $Q$) in terms of ${}^Q\Omega^\pm$ (which define the $Q$-connections on $Q$) and replacing the field $E\in\Gamma(TM\otimes TM)$ with $\sfE:\Gamma(Q\otimes Q)$ defined by
\begin{align*}
\sfE(\cdot,\cdot ):=E(\rho(\cdot),\rho(\cdot)).
\end{align*} 

Note that Equation \eqref{Qcomponents} gives $\rho^\mu_c(\Omega^\pm)^a{}_{\mu b}=({}^Q\Omega^\pm)^a{}_{b c}-C^{a}{}_{bc}$; substituting this expression into Equation \eqref{Eq:ProjGenKilling} gives 
\begin{align}
\rho^\mu_d\rho^\nu_c(\cL_{\rho_a}E)_{\mu\nu}=&\rho^\mu_dE_{\mu\lambda}\rho^\lambda_b(({}^Q\Omega^+)^b{}_{ac}-C^b{}_{ac})+\rho^\lambda_bE_{\lambda\nu}\rho^\nu_c(({}^Q\Omega^-)^b{}_{ad}-C^b{}_{ad}),\nonumber\\
=&\sfE_{db}(({}^Q\Omega^+)^b{}_{ac}-C^b{}_{ac})+\sfE_{bc}(({}^Q\Omega^-)^b{}_{ad}-C^b{}_{ad}).\label{Eq:LiftingKilling}
\end{align}
The Lie algebroid on $TM$ defined by $[\rho_a,\rho_b]=C^c{}_{ab}\rho_c$ can be lifted to a Lie algebroid on $Q$ as follows: Choose a frame for $Q$, denoted $\{e_a \}$, and define a Lie algebroid by $[e_a,e_b]_Q:=C^c{}_{ab}e_c$, and $\rho(e_a):=\rho_a$.

Using a Lie algebroid structure $([\cdot,\cdot]_Q,\rho)$ we can define a Lie derivative on $Q$,
\begin{align*}
(\scL_q\sfE)(q_1,q_2):=\rho(q)(\sfE(q_1,q_2))-\sfE([q,q_1]_Q,q_2)-\sfE(q_1,[q,q_2]_Q),
\end{align*}
for any $q,q_1,q_2\in\Gamma(Q)$ and $\sfE\in\Gamma(Q^*\otimes Q^*)$. The expression in the local basis $\{e_a \}$ is
\begin{align*}
(\scL_q\sfE)(q_1,q_2)=q^a_1q^b_2(q^c\rho^\mu_c\partial_\mu\sfE_{ab}-C^d{}_{ca}\sfE_{db}q^c-C^d{}_{cb}\sfE_{ad}q^c+\sfE_{cb}\rho^\mu_a\partial_\mu q^c+\sfE_{ac}\rho^\mu_b\partial_\mu q^c),
\end{align*}
and it follows that
\begin{align}\label{Eq:CoordQLieDer}
(\scL_{e_a}\sfE)_{dc}=\rho^\mu_a\partial_\mu\sfE_{dc}-C^b{}_{ad}\sfE_{bc}-C^b{}_{ac}\sfE_{db}.
\end{align}
The Lie algebroid Lie derivative acting on $\sfE$ can be related to the standard Lie derivative on $E$: 
\begin{align*}
(\scL_{e_a}\sfE)_{dc}=&\rho^\lambda_a\partial_\lambda(E_{\mu\nu}\rho^\mu_d\rho^\nu_c)-C^b{}_{ad}\rho^\mu_bE_{\mu\nu}\rho^\nu_c-C^b{}_{ac}\rho^\mu_dE_{\mu\nu}\rho^\nu_b,\\
=&\rho^\lambda_a\partial_\lambda(E_{\mu\nu})\rho^\mu_d\rho^\nu_c+E_{\mu\nu}(\rho^\lambda_a\partial_\lambda\rho^\mu_d)\rho^\nu_c+E_{\mu\nu}\rho^\mu_d\rho^\lambda_a\partial_\lambda\rho^\nu_c\\
&-([\rho_a,\rho_d])^\mu E_{\mu\nu}\rho^\nu_c-([\rho_a,\rho_c])^\nu \rho^\mu_dE_{\mu\nu}\\
=&\rho^\lambda_a\partial_\lambda(E_{\mu\nu})\rho^\mu_d\rho^\nu_c+(\rho^\lambda_d\partial_\lambda\rho^\mu_a)E_{\mu\nu}\rho^\nu_c+(\rho^\lambda_c\partial_\lambda\rho^\nu_a)\rho^\mu_dE_{\mu\nu}\\
=&\rho^\mu_d\rho^\nu_c(\cL_{\rho_a}E)_{\mu\nu}.
\end{align*}
Combining \eqref{Eq:LiftingKilling} and \eqref{Eq:CoordQLieDer} gives 
\begin{align}\label{Eq:LiftedGenKilling}
\rho^\mu_a\partial_\mu \sfE_{dc}=\sfE_{db}({}^Q\Omega^+)^b{}_{ac}+\sfE_{bc}({}^Q\Omega^-)^b{}_{ad}.
\end{align}
We conclude that \eqref{Eq:LiftedGenKilling} holds if
$\{\rho_a\}$ satisfy the generalised Killing condition \eqref{Ecompatibility}. 

We would like to choose a frame in which $({}^Q\Omega^-)^b{}_{ad}=0$ (such a frame exists as ${}^Q\Omega^-$ is flat).  In this case \eqref{Eq:LiftedGenKilling} simplifies to
\begin{align}\label{Eq:LiftIdentity}
\rho^\mu_a\partial_\mu \sfE_{dc}=\sfE_{db}({}^Q\Omega^+)^b{}_{ac}.
\end{align} 
As ${}^Q\Omega^+$ is flat we can substitute \eqref{FormofQflat} into \eqref{Eq:LiftIdentity} and conclude that
\begin{align}\label{Eq:KEconstr}
\d{}_Q\sfE=\sfE K^{-1}\d{}_QK,\quad K\in C^\infty(M,\sfGL(k)).
\end{align}
Now we wish to construct a $K$ satisfying Equation \eqref{Eq:KEconstr}.  If $\rank(\sfE)=k$ we can simply choose $K=\sfE$. If $\rank(\sfE)=j<k$ we can still construct $K$.  There exists a frame where $\sfE$ is in the form 
\begin{align}\label{Eq:ReducedE}
\sfE=\left(\begin{matrix}
\sfE' & X'\\ 0 & 0
\end{matrix}\right),
\end{align}
where $\sfE'\in C^\infty(M,\RR^{j\times j})$ has rank $j$, and  $X'\in C^\infty(M,\RR^{j\times (k-j)})$.  Let us denote such a frame by $\{\tilde{e}_a \}$ where $\tilde{e}_a=N^b{}_ae_b$ for some $N\in C^\infty(M,\sfGL(k))$ and $\{e_a \}$ is the original frame defined by $\rho(e_a)=\rho_a$. We can now take  
\begin{align}\label{Eq:FormofK}
K=\left(\begin{matrix}
\sfE' & X'\\ 0 & I
\end{matrix}\right)\quad \Longrightarrow\quad {}^Q\Omega^+=\left(\begin{matrix}
\sfE'^{-1}\d{}_Q\sfE' & \sfE'^{-1}\d{}_QX'\\ 0 & 0
\end{matrix}\right).
\end{align}
A straightforward calculation shows that this choice satisfies \eqref{Eq:KEconstr}. After this choice is made it is possible to change back to the original frame.

The construction given above will produce a valid choice ${}^Q\Omega^+$ with ${}^Q\Omega=0$ (when Theorem \ref{Th:GaugeThrm} holds) in the frame where $\sfE$ is of the form \eqref{Eq:ReducedE}. Transforming back to the original frame will give $({}^Q\Omega)^a{}_{bc}=N^a{}_x\rho^\mu_b\partial_\mu(N^{-1})^x{}_c$. So we need to find a set of vector fields $\{\rho_a \}$ which satisfy Theorem \ref{Th:GaugeThrm} when $\Omega^-$ satisfies $\rho^\mu_c(\Omega^-)^a{}_{\mu b}=C^a{}_{cb}+{}^Q\Omega^a{}_{bc}$.  We are now ready to state the gauging theorem for arbitrary involutive vector fields.

\begin{theorem}\label{Thrm:GaugeArbitrary}
	Given an invertible field $E\in\Gamma(T^*M\otimes T^*M)$ and arbitrary involutive vector fields $\rho_a\in\Gamma(TM)$, $a=1,\dots,k$, defining a Lie algebroid $[\rho_a,\rho_b]=C^c{}_{ab}\rho_c$, the generalised Killing equation
	\begin{align*}
	(\cL_{\rho_a}E)_{\mu\nu}=&E_{\mu\lambda}\rho^\lambda_b(\Omega^+)^b{}_{\nu a}+E_{\lambda\nu}\rho^\lambda_b(\Omega^-)^b{}_{\mu a},
	\end{align*}
	has solutions defining flat adjoint connections ${}^Q\nabla^\pm$ if  
	\begin{align}\label{eq:GaugingConditions2}
	\Psi'_a=E\rho(E\rho)^+\Psi'_a,
	\end{align}
	where 
	\begin{align}
	(\Psi'_a)_{\mu\nu}:=(\cL_{\rho_a}E)_{\mu\nu}-E_{\lambda\nu}\rho^\lambda_b(\Omega^-)^b_{\mu a},
	\end{align}
	for some $(\Omega^-)^b_{\mu a}$ which satisfies $\rho^\mu_c(\Omega^-)^b_{\mu a}=C^b{}_{ca}+({}^Q\Omega)^b{}_{ac}$ (where ${}^Q\Omega^b{}_{ac}$ is given by \eqref{eq:Omegasolns2}).
	Let $N\in C^\infty(M,\sfGL(k))$ be such that the frame $\tilde{e}_a=N^b{}_ae_a$ brings $\sfE$ to the form \eqref{Eq:ReducedE}. If these conditions are satisfied then there exist ${}^Q\nabla^\pm$ defined by 
	\begin{align}\label{eq:Omegasolns2}
	({}^Q\Omega^-)^a{}_{bc}=N^a{}_x\rho^\mu_b\partial_\mu(N^{-1})^x{}_c,\quad\quad ({}^Q\Omega^+)^a{}_{bc}=((KN^{-1})^{-1})^a{}_x\partial^\mu_b(KN^{-1})^x{}_c,
	\end{align}
	where $K\in C^\infty(M,\sfGL(k))$ is constructed following the discussion preceding this theorem (Equation \eqref{Eq:FormofK}).
\end{theorem}	

The lifting perspective of Lie algebroid gauging is useful even when the vector fields are linearly independent.  In fact, when $\rho$ is invertible (establishing an isomorphism between $Q$ and $TM$) we find that the generalised Killing equation can be solved for any choice of flat ${}^Q\nabla^-$ while simultaneously satisfying the gauge algebroid.
\begin{lemma}\label{Lemma:Curvature}
	Given a connection ${}^Q\nabla^-$ defined in a local frame by $({}^Q\Omega^-)^a{}_{bc}$, and ${}^Q\nabla^+$ defined by
	\begin{align}\label{Eq:InvertibleSolution}
	({}^Q\Omega^+)^b{}_{ac}=(\sfE^{-1})^{bd}\rho^\mu_a\partial_\mu \sfE_{dc}-\sfE_{dc}({}^Q\Omega^-)^d{}_{ae}(\sfE^{-1})^{be},
	\end{align} 
	the following identity holds:
	\begin{align}
	(R_{{}^Q\nabla^+})^d{}_{abc}=-\sfE_{ec}(R_{{}^Q\nabla^-})^e{}_{abf}(\sfE^{-1})^{df}.
	\end{align}
\end{lemma}
\begin{proof}
	The proof of this is simply a computation and we give details in the Appendix.
\end{proof}
\begin{corollary}\label{Thrm:FlatConnectionChoice}
Given an invertible $\rho$, the generalised Killing equation \eqref{Ecompatibility} has solutions. 
Furthermore, if ${}^Q\Omega$ defines a flat connection ${}^Q\nabla^-$ it follows that ${}^Q\Omega$ defines a flat connection ${}^Q\nabla^+$.
\end{corollary}
\begin{proof}
	The statement that \eqref{Eq:InvertibleSolution} satisfies the generalised Killing equation is easily verified.  It is always possible to choose $({}^Q\Omega^-)^a{}_{bc}=0$, so a solution always exists. The final statement follows from Lemma \ref{Lemma:Curvature}.	
\end{proof}
\subsection{Lie algebroid gauging examples}\label{Sec:Examples}~\\
We conclude the derivation of the general solution to the Lie algebroid gauging conditions (Equations \eqref{Ecompatibility} and \eqref{Eq:FlatnessConstraint}) with examples demonstrating the results.  Theorems \ref{Thrm:GaugeEverything} and \ref{Thrm:GaugeArbitrary} give a methodical procedure for determining whether a set of vector fields can be used to gauge an action---as well as constructing a choice of connection coefficients when gauging is possible.  The first example provides an explicit use of the results developed in this section. The second example shows that the gauging procedure (when it is defined) is not unique. The third example gives an alternative choice of gauging.
\begin{example}\label{ex:Gauging}
	Let $M=\RR^3$ with coordinates $\{x,y,z\}$.  Take 
	\begin{align*}
	G=(\d x)^2+(\d y)^2+(1+x^2)(\d z)^2,\quad C=2x\d x\wedge \d z,
	\end{align*}
	giving
	\begin{align*}
	E=\left(\begin{matrix}
	1 & 0 & x\\
	0 & 1 & 0\\
	-x & 0 & 1+x^2
	\end{matrix}\right),\quad H=\d C=0.
	\end{align*}
	The field $E\in\Gamma(T^*M\otimes T^*M)$ is invertible as $\det(E)=1+2x^2\neq 0$.
	
	Consider the possibility of gauging with respect to the vector fields $\{\rho_a \}=\{\partial_x,\partial_y \}$.  
	\begin{align*}
	\cL_{\partial_x}E=\left(\begin{matrix}
	0 & 0 & 1\\
	0 & 0 & 0\\
	-1 & 0 & 2x
	\end{matrix}\right),\quad \cL_{\partial_y}E=0,\quad \rho=\left(\begin{matrix}
	1 & 0 \\
	0 & 1 \\
	0 & 0 
	\end{matrix}\right),\quad \rho^+=\left(\begin{matrix}
	1 & 0 & 0\\
	0 & 1 & 0
	\end{matrix}\right).
	\end{align*}
	We note that $\cL_{\partial_x}E\neq0$ so $\partial_x$ is not a Killing vector field.  The vector fields $\{\partial_x,\partial_y \}$ are linearly independent---so we apply Theorem \ref{Thrm:GaugeEverything}. In this case the necessary consistency condition \eqref{eq:NeccesaryConstistency} is not satisfied:
	\begin{align*}
	(E\rho)(E\rho)^+=\frac{1}{1+x+x^2}\left(\begin{matrix}
	x & 0 & 1-2x^2\\
	0 & 0 & 0\\
	-x & 0 & 2x^3-1
	\end{matrix}\right)\quad \Longrightarrow \quad \cL_{\partial_x}E \neq E\rho(E\rho)^+\cL_{\partial_x}E.
	\end{align*}
	We conclude that it is not possible to gauge $E$ non-isometrically using the vector fields $\{\partial_x,\partial_y \}$.  
	
	Now consider the possibility of gauging with respect to $\{\rho'_a \}=\{\partial_x,\partial_z \}$.
	\begin{align*}
	\cL_{\partial_x}E=\left(\begin{matrix}
	0 & 0 & 1\\
	0 & 0 & 0\\
	-1 & 0 & 2x
	\end{matrix}\right),\quad \cL_{\partial_z}E=0,\quad \rho'=\left(\begin{matrix}
	1 & 0 \\
	0 & 0 \\
	0 & 1 
	\end{matrix}\right),\quad \rho'^+=\left(\begin{matrix}
	1 & 0 & 0\\
	0 & 0 & 1
	\end{matrix}\right).
	\end{align*}
	In this case the necessary consistency conditions are satisfied, with the only non-trivial part being
	\begin{align*}
	\cL_{\partial_x}E=\left(\begin{matrix}
	0 & 0 & 1\\
	0 & 0 & 0\\
	-1 & 0 & 2x
	\end{matrix}\right)= \left(\begin{matrix}
	1 & 0 & 0\\
	0 & 0 & 0\\
	0 & 0 & 1
	\end{matrix}\right)\left(\begin{matrix}
	0 & 0 & 1\\
	0 & 0 & 0\\
	-1 & 0 & 2x
	\end{matrix}\right)=E\rho'(E\rho')^+\cL_{\partial_x}E.
	\end{align*} 
	We conclude that we can gauge with respect to $\{\partial_x,\partial_z \}$.  Explicitly, using \eqref{eq:Omegasolns}, we find
	\begin{align*}
	\Omega^-=0,\quad \Omega^+=\frac{1}{1+2x^2}\left(\begin{matrix}
	x\d x+(1-x^2)\d z & 0\\ 
	-\d x+3x\d z & 0
	\end{matrix}\right)\quad \Longrightarrow\quad {}^Q\Omega^-=0,\quad {}^Q\Omega^+=\sfE^{-1}\partial_x\sfE \d x,
	\end{align*}
	where
	\begin{align*}
	\sfE=\rho^TE\rho=\left(\begin{matrix}
	1 & x\\ -x & 1+x^2
	\end{matrix}\right),\quad \det(\sfE)=1+2x^2\neq 0.
	\end{align*}
\end{example}

Example \ref{ex:Gauging} provides us with an explicit example of gauging where $R_{\nabla^+}\neq 0$ but the gauge algebroid closes ($R_{{}^Q\nabla^\pm}=0$).  A direct calculation gives
\begin{align*}
R_{\nabla^+}=\frac{1}{2(1+2x^2)^2}\left(\begin{matrix}
-2x(5+x^2) & 0\\
5-8x^2 & 0
\end{matrix}\right)\d x\wedge \d z.
\end{align*}
In this case $\nabla^+$ does not define a representation of a Lie algebroid on $Q\cong T\cF$ (where the leaves of $\cF$ are the $xz$-planes).  The local gauging data is given by $\{\rho_a,\Omega^\pm \}$. However, from a geometric perspective the gauging symmetry should not be viewed as arising from $\nabla^\pm$.  The vector fields $\{\rho_a \}$ generate the action of the gauging symmetry on $TM$.  The connections ${}^Q\Omega^\pm$ describe the lifted action on sections of $Q$.  The gauging is associated to the flow of the Lie algebroid actions defined by ${}^Q\nabla^\pm$.  The flat connections ${}^Q\nabla^\pm$ define representations of two Lie algebroids $(Q,{}^Q\nabla^\pm)$.  The infinitesimal action is generated by $(Q,{}^Q\nabla^\pm)$.  The finite groupoid action comes from the Weinstein Lie groupoids $\cG({}^Q\nabla^\pm)$.  

\begin{example}
	The choice of gauging in Example \ref{ex:Gauging} is not unique.  Example \ref{ex:Gauging} can also be gauged with respect to the vector fields $\{\rho''_a \}=\{\partial_x,\partial_y,\partial_z \}$.  In this case one can take 
	\begin{align*}
	{}^Q\Omega^+=E^{-1}\partial_x E\d x-E({}^Q\Omega^-)E^{-1},
	\end{align*}
	where ${}^Q\Omega^-$ is any choice of $Q$-flat connection. 
\end{example}

Poisson--Lie T-duality, introduced by Klim\v{c}\'{i}k and \v{S}evera \cite{Kli95,Kli95a}, describes an equivalence\footnote{The equivalence here is a symplectomorphism between the phase spaces of both models.} between two non-linear sigma models. 

The non-linear sigma models described by Poisson--Lie T-duality are of the form
\begin{align}
S[X]=\int dzd\bar{z}(G_{\mu\nu}+B_{\mu\nu})\partial X^\mu\bar{\partial} X^\nu=:\int dz d\bar{z}E_{\mu\nu}\partial X^\mu \bar{\partial} X^\nu.
\end{align}
Suppose that there is a right-action of a Lie group $\sfG$ on the closed target manifold $M$.  Choosing a basis of left-invariant vector fields $v_a\in\Gamma(TM)$, $a=1,\dots,\dim(\sfG)$, the Lie algebra $\fg=\Lie(\sfG)$ is realised by
\begin{align*}
[v_a,v_b]=C^c{}_{ab}v_c,
\end{align*} 
where $C^c{}_{ab}\in\RR$ are structure constants. The vector fields $\rho_a$ are not isometries of $E$. Instead the following Poisson--Lie condition is required to hold:
\begin{align}\label{PLcompatibility}
(\cL_{\rho_a}E)_{\mu\nu}=\widetilde{C}^{kl}{}_{a}\rho^\lambda_k\rho^\tau_lE_{\lambda\nu}E_{\mu\tau},
\end{align}
for some constants $\widetilde{C}^{ab}{}_c$.  The identity $[\cL_{\rho_a},\cL_{\rho_b}]E=\cL_{[\rho_a,\rho_b]}E$ imposes the constraint 
\begin{align}\label{Eq:LieBialgebra}
2C^d{}_{[a|g|}\widetilde{C}^{gl}{}_{b]}+2\widetilde{C}^{dg}{}_{[b}C^f{}_{a]g}-C^g{}_{ab}\widetilde{C}^{dl}{}_{g}=0.
\end{align}
This constraint is the relation for the structure constants of the Lie bialgebra $(\sfG,\widetilde{\sfG})$.\footnote{For a review of Drinfeld doubles and Lie bialgebras see for example \cite{Ale94,Kos07}.}  

We claim that a non-linear sigma model satisfying the Poisson--Lie gauging conditions can be gauged for a choice ${}^Q\Omega^\pm$ that differs from that of Theorem \ref{Thrm:GaugeEverything}. This highlights the lack of uniqueness of the Lie groupoid gauging procedure.
\begin{example}
Consider a Poisson--Lie sigma model i.e. some $E\in\Gamma(T^*M\otimes T^*M)$ satisfying \eqref{PLcompatibility} where $C^a{}_{bc}$ and $\widetilde{C}^{ab}{}_c$ satisfy \eqref{Eq:LieBialgebra}. Choose the frame $\{e_a \}$ where $\rho_a:=\rho(e_a)$ coincides with the vector fields for Poisson--Lie gauging. Take
\begin{align*}
({}^Q\Omega^+)^a{}_{bc}=C^a{}_{bc},\quad\quad ({}^Q\Omega^-)^a{}_{bc}=\widetilde{C}^{ad}{}_b\sfE_{cd}+C^a{}_{bc},
\end{align*}
where $\sfE=E(\rho(\cdot),\rho(\cdot))$. The Lie algebroid gauging condition \eqref{Ecompatibility} is equivalent to Equation \eqref{PLcompatibility}. Calculating the $Q$-curvatures we find:
\begin{subequations}
	\begin{align}
	(R_{{}^Q\nabla^+})_{abc}=&\rho^\mu_a\partial_\mu C^d{}_{bc}e_d;\\
	(R_{{}^Q\nabla^-})_{abc}=&\Big(\rho^\mu_a\partial_\mu C^d{}_{bc}+2\rho^\mu_{[a}\partial_{|\mu|}(\widetilde{C}^{dl}{}_{b]}\sfE_{cl})+2(C^d{}_{[a|g|}\widetilde{C}^{gl}{}_{b]}-C^g{}_{ab}\widetilde{C}^{dl}{}_{g})\sfE_{cl}\label{PLR+}\\
	&+2\widetilde{C}^{dl}{}_{[a}C^g{}_{b]c}\sfE_{gl}+2\widetilde{C}^{dk}{}_{[a|}\sfE_{gk}\widetilde{C}^{gl}{}_{|b]}\sfE_{cl}\Big)e_d.\nonumber
	\end{align}
\end{subequations}
It is clear that $R_{{}^Q\nabla^+}=0$, as $C^a{}_{bc}\in\RR$.  To see that $R_{{}^Q\nabla^-}$ vanishes requires a little bit more work.  Using the definition of the Lie derivative, and \eqref{PLcompatibility}, we have
\begin{align*}
\rho^\lambda_a\partial_\lambda E_{\mu\nu}=\widetilde{C}^{kl}_a\rho^\lambda_k\rho^\tau_lE_{\lambda\nu}E_{\mu\tau}-(\partial_\mu\rho^\lambda_a)E_{\lambda\nu}-(\partial_\nu\rho^\lambda_a)E_{\mu\lambda}.
\end{align*}
Multiplying by $\rho^\mu_b\rho^\nu_c$ (which are invertible):
\begin{align*}
\rho^\mu_b\rho^\nu_c\rho^\lambda_a\partial_\lambda E_{\mu\nu}=&\widetilde{C}^{kl}_a\sfE_{kc}\sfE_{bl}-(\rho^\mu_b\partial_\mu\rho^\lambda_a)E_{\lambda\nu}\rho^\nu_c-(\rho^\nu_c\partial_\nu\rho^\lambda_a)E_{\mu\lambda}\rho^\mu_b\\
=&\widetilde{C}^{kl}_a\sfE_{kc}\sfE_{bl}+C^d{}_{ab}\rho^\lambda_{d}E_{\lambda\nu}\rho^\nu_c+(\rho^\mu_a\partial_\mu\rho^\lambda_b)E_{\lambda\nu}\rho^\nu_c\\
&+C^d{}_{ac}\rho^\lambda_dE_{\mu\lambda}\rho^\mu_b+(\rho^\nu_a\partial_\nu\rho^\lambda_c)E_{\mu\lambda}\rho^\mu_b,
\end{align*}
so that $\rho^\mu_a\partial_\mu \sfE_{bc}=\widetilde{C}^{kl}_a\sfE_{kc}E_{bl}+C^d{}_{ab}\sfE_{dc}+C^d{}_{ac}\sfE_{bd}$.  Substituting this into \eqref{PLR+}, gives
\begin{align*}
(R^{{}^Q\nabla^-})_{abc}=(2C^d{}_{[a|g|}\widetilde{C}^{gl}{}_{b]}+2\widetilde{C}^{dg}{}_{[b}C^f{}_{a]g}-C^g{}_{ab}\widetilde{C}^{dl}{}_{g})\sfE_{cl}e_d.
\end{align*}
We conclude that $R_{{}^Q\nabla-}=0$ as $C^a{}_{bc}$ and $\widetilde{C}^{ab}{}_{c}$ define a Lie bialgebra.
\end{example}


\section{Recovering the ungauged action}\label{Sec:RecoverUngauged}
This section discusses the difficulty in understanding when the ungauged model can be recovered from the gauged model.  

If we wish to claim that the gauged action is equivalent to the ungauged action there must exist a finite gauge transformation which sets the gauge field $A$ to zero.  If $A$ is in the gauge orbit of zero then we can recover the original action by choosing the $A=0$ gauge. If $A$ is not gauge equivalent to zero then the gauged action is not equivalent to the original action.  In the case of the Wess--Zumino--Witten (WZW) model the finite action can be explicitly written down.  This is briefly reviewed in Section \ref{Sec:WZWmodel}. A Gauge field $A$ in the gauge orbit of zero must satisfy the Maurer--Cartan equation. This condition can be imposed via a Lagrange multiplier term.  The addition of this gauge invariant term is essential for some applications (such as T-duality).  

In the more general case of Lie algebroid gauging we should require that the gauge field $A$ is gauge equivalent to zero through some Lie groupoid action. The Maurer--Cartan equation for Lie groupoids is discussed in Section \ref{Sec:MaurerCartan}. In Section \ref{Sec:SpecialCaseF} we see that in the special case that $\Omega^+=\Omega^-$ there is a corresponding Maurer--Cartan equation when $R_{\nabla^{\pm}}=0$. However, in this case the Lie groupoid Maurer--Cartan equation is isomorphic to the usual Maurer--Cartan equation and the applicable Lie algebroids are (locally) isomorphic to Lie algebras. Finally Section \ref{Sec:GenMaurerCartan} discusses the fact that we do not expect a general field strength term to be able to be written in terms of our gauge fields $A$.

\subsection{WZW model}\label{Sec:WZWmodel}~\\
The standard WZW model describes the embedding of a string worldsheet $\Sigma$ into a Lie group $\sfG$ via $g:\Sigma\rightarrow \sfG$.  The non-linear sigma model is given by
\begin{align*}
S_{\text{WZW}}[g]=\frac{1}{2}\int_{\Sigma}\d {}^2z (g^{-1}\partial g)^\mu E_{\mu\nu} (g^{-1}\bar{\partial}g)^\nu,
\end{align*}
where $E_{\mu\nu}\in\RR$.  The action is invariant under the left action $hg$ for $h\in\sfG$.  It is possible to promote this to a local action $h\in C^\infty(\Sigma,\sfG)$ by introducing a gauge field $A\in\Omega^1(\Sigma,\fg)$.  The gauged action is
\begin{align*}
S_{\text{WZW}}[g,A]=\frac{1}{2}\int_{\Sigma}\d {}^2z (g^{-1}D g)^\mu E_{\mu\nu} (g^{-1}\bar{D}g)^\nu,
\end{align*}
where $g^{-1}Dg=g^{-1}\d g-g^{-1}Ag$. The gauge transformations are given by
\begin{align*}
{}^h(g,A)=(hg,hAh^{-1}+\d h h^{-1}).
\end{align*}
The original action $S_{\text{WZW}}[g]$ can be recovered from the gauged action $S_{\text{WZW}}[g,A]$ if we can set $A=0$:
\begin{align*}
0={}^hA=hAh^{-1}+\d h h^{-1}\quad \Longrightarrow \quad A=-h^{-1}\d h,
\end{align*}
for some $h\in C^\infty(\Sigma,\sfG)$.\footnote{If $A$ is in the gauge orbit of zero, i.e. $A=-h^{-1}\d h$ for some $h\in C^\infty(\Sigma,\sfG)$, the gauged model is equivalent to the original via the field redefinition $g\rightarrow hg$.} It is well known that $A=-h^{-1}dh$ is the most general solution to the Maurer--Cartan equation $F=0$ where $F\in\Omega^2(\Sigma,\sfG)$
\begin{align}
F=\d A-[A \wedgec A]_\fg,\quad \quad \left(F^a=\d A^a+\frac{1}{2}C^a{}_{bc}A^b\wedge A^c \right).
\end{align}
It is common to refer to $F$ as the \emph{field strength}.  We conclude that the ungauged action is equivalent to the gauged model if and only if we impose the constraint $F=0$.  It is possible to consider this as an extra constraint. However, it is often useful to include it as an equation of motion.  To do this we add a Lagrange multiplier term to the action
\begin{align*}
S_F[A,\widetilde{X}]=\int_\Sigma \langle\widetilde{X},F\rangle:=\int_\Sigma \widetilde{X}_aF^a,
\end{align*}
where $\widetilde{X}\in C^\infty(\Sigma,\fg^*)$.  The constraint $F=0$ comes from the equations of motion for $\widetilde{X}$.   

If we wish to add the term $S_F[A,\widetilde{X}]$ to the action $S_{\text{WZW}}[g,A]$ then it should preserve gauge invariance. Under a gauge transformation of $h\in C^\infty(\Sigma,\sfG)$ the fields strength transforms as
\begin{align*}
{}^hF=\Ad_h F:=hFh^{-1}.
\end{align*}
The term in the Lagrangian will be invariant under the left action of $h\in C^\infty(\Sigma,\sfG)$ if and only if
\begin{align*}
{}^h\widetilde{X}=\Ad^*_{h^{-1}}\widetilde{X},
\end{align*}
where $\Ad^*_{h^{-1}}$ is the \emph{coadjoint action} $\langle \Ad^*_{h}\widetilde{X},X\rangle:=\langle\widetilde{X},\Ad_hX\rangle$ for $\widetilde{X}\in \fg^*$ and $X\in \fg$. The natural question now is how much of this generalises to the case of Lie algebroid gauging?  In the special case where $\Omega^+=\Omega^-$ we will see in Section \ref{Sec:SpecialCaseF} that a field strength term exists. However, gauge invariance requires that the Lie algebroid is in fact isomorphic to a Lie algebra.

\begin{remark}
	The introduction of the curvature term is an essential part of the non-abelian T-duality procedure (for a description of non-abelian T-duality see \cite{Roc91,del92}).  The Lagrange multipliers $\widetilde{X}$ used to impose $F=0$ in the original model are interpreted as local coordinates on some dual manifold $\widetilde{M}$ on which a dual non-linear sigma model is defined.
\end{remark}  

Even if one chooses to impose the condition that $A$ is in the gauge orbit of zero by hand (avoiding the field strength term) it is still important to have an understanding of the general form of $A$ i.e. $A=-h^{-1}\d h$ for some $h\in C^\infty(\Sigma,\sfG)$.


\subsection{Maurer--Cartan algebroid}\label{Sec:MaurerCartan}~\\
There is a convenient description of certain Lie algebroid morphisms in terms of a Maurer--Cartan form.  This interpretation is due to Fernandes and Struchiner \cite{Fer07}.


Consider a Lie algebroid morphism $(\Phi,X)$ where $\Phi:T\Sigma\rightarrow Q$, and $X:\Sigma\rightarrow M$.  The key observation is to consider $\Phi:\Gamma(T\Sigma)\rightarrow \Gamma(Q)$ as $A\in\Gamma(T^*\Sigma\otimes Q)$:
\begin{align*}
A(s):=\Phi^*(s),
\end{align*}
for $s\in \Gamma(T\Sigma)$.  The maps $(\Phi,X)$ define a Lie algebroid morphism if and only if $A$ can be interpreted as a Maurer--Cartan form. 

The setup can be described with the following commutative diagrams:
\begin{equation*}
\begin{tikzpicture}[node distance=2cm, auto]
\node (TS) {$T\Sigma$};
\node [right of=TS] (Q') {$Q$};
\node [below of=TS] (S) {$\Sigma$};
\node [right of=S] (S2) {$M$};
\node [right of=Q', node distance=4cm] (TS2) {$T\Sigma$};
\node [right of=TS2] (Q'2) {$Q$};
\node [below of=Q'2] (TS3) {$TM$};

\node [below of=Q', node distance=1cm] (A){};
\node [right of=A] (B){such that};

\draw[->] (TS) -- node[above,anchor=east]{} (S); 
\draw[->] (TS) -- node[above, anchor=south]{$A$} (Q');
\draw[->] (S) --node[above, anchor=south]{$X$} (S2);
\draw[->] (Q') -- (S2);
\draw[->] (TS2) -- node[above,anchor=south]{$A$} (Q'2);
\draw[->] (TS2) -- node[above,anchor=east]{$X_*$} (TS3);
\draw[->] (Q'2) -- node[above,anchor=west]{$\rho$} (TS3);
\end{tikzpicture},
\end{equation*}
where the second diagram expresses the anchor compatibility condition.  Given a choice of $T\Sigma$-connection on $Q$, denoted $\nabla$, define
\begin{align}
(\d{}_{\nabla} A)(s_1,s_2):=&\nabla_{s_1} A(s_2)-\nabla_{s_2}A(s_1)-A([s_1,s_2]_{T\Sigma}),\\
[A\wedgec A']_{\nabla}(s_1,s_2)=&\tfrac{1}{2}\left(T_{\nabla}(A(s_1),A'(s_2))+T_{\nabla}(A(s_2),A'(s_1))\right),
\end{align}
where $s_1,s_2\in\Gamma(T\Sigma)$ and $A,A'\in\Gamma(T^*\Sigma\otimes Q)$.
The expression
\begin{align}\label{LieAlgmorphism}
F_A(s_1,s_2)=&(\d{}_{\nabla}A-[A\wedgec A]_{\nabla})(s_1,s_2),
\end{align}
defines an element $F_A\in\Gamma(\wedge^2T^*\Sigma\otimes Q)$ which is independent of the choice of $\nabla$.

An element $A\in\Gamma(T^*\Sigma\otimes Q)\cong\Omega^1(\Sigma,Q)$ satisfies $F_A\equiv 0$ if and only if it defines a Lie algebroid morphism between $T\Sigma$ and $Q$.  In this case $A$ can be interpreted as a Maurer--Cartan form as follows: Let $\cG$ be a Lie groupoid with Lie algebroid $Q$.  Left translation by an element $h\in\cG$ is a diffeomorphism between $\s{}$-fibres $L_h:\s{}^{-1}(\s{}(h))\rightarrow \s{}^{-1}(\t{}(h))$.  A \emph{left-invariant one-form} on $\cG$ is an $\s{}$-foliated one-form $A$ on $\cG$ such that for all $h\in\cG$
\begin{align*}
A(X)=A(\d{}_gL_h(X)),\quad \forall g\in \s{}^{-1}(\s{}(h)),\ X\in T^{\s{}}_g\cG.
\end{align*}  
This is also denoted $(L_h)^*A=A$.  A \emph{Maurer-Cartan form on a Lie groupoid $\cG$} is the $Q$-valued $\s{}$-foliated left-invariant one-form defined by
\begin{align*}
A(X)=(\d L_{h^{-1}})_h(X),
\end{align*}  
for $X\in T^{\s{}}_h\cG$.  The Maurer--Cartan form $A:T^{\s{}}\cG\rightarrow Q$ covers the target map $\t{}:\cG\rightarrow M$.

\begin{theorem}[\cite{Fer07}]\label{GroupoidMorphismThrm}
	Let $\cG$ be a Lie groupoid with Lie algebroid $Q$ and let $A$ be its left invariant Maurer--Cartan form.  If $A':T\Sigma\rightarrow Q$ is a solution of the Maurer--Cartan equation covering a map $X:\Sigma\rightarrow M$, then for each $\sigma\in \Sigma$ and $h\in \cG$ such that $X(\sigma)=\s{}(h)$, there exists a unique locally defined diffeomorphism $X':\Sigma\rightarrow {\s{}}^{-1}(X(\sigma))$ satisfying:
	\begin{align*}
	X'(\sigma)=h,\quad {X'}^*A=A'.
	\end{align*}	
\end{theorem} 

\begin{example}
	Choose a manifold $M$ and a Lie algebra $\fg$. A Lie algebroid morphism $TM\rightarrow \fg$ is the same thing as a  one-form $A\in\Omega^1(M,\fg)$ satisfying the Maurer--Cartan equation $\d A-[A \wedgec A]_\fg=0$.
\end{example}


\subsection{Special case where $\Omega^+=\Omega^-$}\label{Sec:SpecialCaseF}~\\
An appropriate Maurer--Cartan equation imposing the condition that the gauge field $A$ is gauge equivalent to zero (and hence the gauged action is equivalent to the ungauged action) has appeared in the literature for the special case that $\Omega^+=\Omega^-$ (or equivalently $\phi=0$).  The local coordinate formula 
\begin{align}
F^a_{\nabla^\omega}=\d A^a+\omega^a{}_{\mu b}A^b\wedge DX^\mu+\tfrac{1}{2}C^a{}_{bc}A^b\wedge A^c,
\end{align}
was given in \cite{May08}.  This can be identified as a Lie algebroid Maurer--Cartan form:\footnote{Here we ignore issues regarding pullbacks (discussed in Section \ref{ProperPullback}) to compare with the expression given in the literature.}
\begin{align*}
F_{\nabla^\omega}=\d{}_{\nabla^\omega}A-[A\wedgec A]_{\nabla^\omega}.
\end{align*}  
This can be verified by an explicit calculation in local coordinates:
\begin{align*}
(\d{}_{\nabla^\omega}A)(v_1,v_2)=&\nabla^\omega_{v_1}(A(v_2))-\nabla^\omega_{v_2}(A(v_1))-A([v_1,v_2]_{TM})\\
=&v^\mu_1 v^\nu_22(\partial_{[\mu}A^a_{\nu]}+\omega^a{}_{[\mu|b}A^b_{|\nu]})e_a,\\
[A\wedge A]_{\nabla^\omega}(v_1,v_2)=&\tfrac{1}{2}(T_{\nabla^\omega}(A(v_1),A(v_2))+T_{\nabla^\omega}(A(v_2),A(v_1)))\\
=&v^\mu_1v^\nu_2A^b_{[\mu} A^c_{\nu]}(2\rho^\lambda_{[b|}\omega^a{}_{\mu|c]}-C^a{}_{bc})e_a.
\end{align*}
Taking $v^\mu_1=\d X^\mu$ and $v^\nu_2=\d X^\nu$, we have
\begin{align*}
F_{\nabla^\omega}=&(\d A^a+\omega^a{}_{\mu b}A^b\wedge dX^\mu-\omega^a_{\mu b}A^b\wedge \rho^\mu_c A^c+\tfrac{1}{2}C^a{}_{bc}A^b\wedge A^c)e_a\\
=&(\d A^a+\omega^a{}_{\mu b}A^b\wedge DX^\mu+\tfrac{1}{2}C^a{}_{bc}A^b\wedge A^c)e_a.
\end{align*}

The term $F_{\nabla^\omega}=0$ will define a Lie algebroid Maurer--Cartan equation if it transforms covariantly under gauge transformations i.e. $\delta_{\varepsilon}F_{\nabla^\omega}\propto F_{\nabla^\omega}$ (the infinitesimal Lie algebroid analogue of ${}^hF=\Ad_h F$).  The infinitesimal variation can be calculated directly:
\begin{align*}
\delta_{\varepsilon}(F_{\nabla^\omega})^a=(C^a{}_{bc}-\rho^\mu_b\omega^a{}_{c\mu})\varepsilon^c(F_{\nabla^\omega})^b+(R_{\nabla^\omega})^a{}_{b\mu\nu}\varepsilon^bDX^\mu\wedge DX^\nu+D^a{}_{bc\mu}\varepsilon^cDX^\mu\wedge A^b,
\end{align*}
where $D^a{}_{bc\mu}=(\nabla^\omega_{\mu}T_{\nabla^\omega})^a{}_{bc}-2\rho^\nu_{[b}(R_{\nabla^\omega})^a{}_{c]\nu\mu}$. It follows immediately that we require $R_{\nabla^\omega}=D^a{}_{bc\mu}=0$. This requires $\omega$ to be a flat $TM$-connection on $Q$.

If $R_{\nabla^\omega}=0$ it follows from the results of \cite{Bou17} that $\cG(Q^\omega)$ is isomorphic to a Lie group $\sfG$.  In the frame where $\omega^a{}_{\mu b}=0$ the $Q$-paths coincide with the flowlines along the right-invariant sections $\fX_{\text{inv}}(\sfG)$.

The condition $F_{\nabla^\omega}=0$ is equivalent to the statement that $A:T\Sigma\rightarrow Q$ is a Lie algebroid morphism. In particular, we can interpret $A$ as a Maurer--Cartan form for the Lie groupoid $\cG(Q^\omega)$ (where $Q^{\omega}$ denotes the Lie algebroid defined by the representation $\nabla^\omega$).  The groupoid is specified by flowing along paths defined on the Lie algebroid $Q^\omega$.

\subsection{General case $R_{{}^Q\nabla^\pm}=0$}\label{Sec:GenMaurerCartan}~\\
In the case that $\Omega^+=\Omega^-$ ($\phi=0$) the field strength term $F_{\nabla^\omega}$ has the interpretation as the Maurer--Cartan equation for the gauge field $A$ and can be interpreted as a Lie algebroid morphism. This suggests that a field strength term for Lie groupoid gauging would be based on gauge fields $\cA^\pm\in\Gamma((X^{**}Q)^*\otimes Q)$ describing a Lie algebroid morphism between Lie algebroids on the vector bundles $X^{**}Q$ and $Q$.  The flat connections ${}^Q\nabla^\pm$ define Lie algebroids, however the gauging procedure does not define any fields which could be interpreted as the required Maurer--Cartan forms $\cA^\pm$.  This suggests that it is not possible to find a Maurer--Cartan equation for the fields $A$ which will be equivalent to $A$ being in the gauge orbit of zero. It is important to understand which $A$ result in a gauge equivalent action and which do not.  Without the knowledge of the finite groupoid action on $A$ the usefulness of the Lie algebroid gauging procedure is potentially limited.

\section{Comments on groupoid structure}\label{Sec:Groupoid}

\begin{definition}\label{Def:LocalGauging}
	A non-linear sigma model $(X,\Sigma,h,E,H,S_Q[X])$  can be \emph{locally gauged} when there exists some $U\subset M$ such that $\rho_a\in\Gamma(TU)$ and ${}^Q\Omega^\pm\in\Gamma(Q^*_U\otimes \End(Q_U))$ satisfy the requirements of Theorem \ref{Thrm:GaugeEverything}.  The local non-linear sigma model is given by the restriction of $(X,\Sigma,h,E,H,S_Q[X])$ to $U\subset M$. The adjoint connections ${}^Q\nabla^\pm$ give representations of local algebroid actions. If $U=M$, and the Lie algebroid actions can be integrated to Lie groupoid actions $\cG({}^Q\nabla^\pm)$, we say that the non-linear sigma model can be \emph{gauged}.
\end{definition}


\begin{corollary}\label{Cor:LocalGauging}
	All non-linear sigma models  $(X,\Sigma,h,E,H,S_Q[X])$ can be locally gauged when $E\in \Gamma(T^*M\otimes T^*M)$ is invertible.
\end{corollary}
\begin{proof}
	Choose a local trivialisation of $M$, denoted $U\subset \RR^n$, and a set of linearly independent $\{\rho_a \}\in\Gamma(TM)$, $a=1,\dots,n=\dim(M)$ (such a set of vector fields exists see for Example \ref{Ex:LocalGauging}). These vector fields define a local frame for $Q_U\cong TU$ and define a Lie algebroid $[\rho_a,\rho_b]=C^c{}_{ab}\rho_c$.  The associated matrix $\rho$ is square and invertible, giving $\rho^+=\rho^{-1}$.  If $E$ is invertible then $E\rho$ and $\rho^TE\rho$ have $\rank(n)$ and $(E\rho)^+=(E\rho)^{-1}$.  The compatibility conditions \eqref{eq:GaugingConditions} are satisfied.  The construction of Theorem \ref{Thrm:GaugeEverything} gives a choice of $\nabla^\pm$ (or equivalently ${}^{Q_U}\nabla^\pm$).
\end{proof}

\begin{example}[Local gauging]\label{Ex:LocalGauging}
	Take a manifold $M$, and some $U\subset M$ with local coordinates $\{x_1,x_2,\dots,x_n \}$.  It is always possible to take $Q_U\cong TU$ and $\{\rho_a\}=\{\partial_{x_1},\partial_{x_2},\dots,\partial_{x_n} \}$ giving $\rho=I_{n\times n}$.  
\end{example}

\begin{corollary}[Global gauging]\label{GlobalGauging}
	If there exists a global frame (the manifold $M$ is parallelisable) the non-linear sigma model $(X,\Sigma,h,E,H,S_Q[X])$ can be gauged.
\end{corollary}
\begin{proof}
	Take a set of globally defined non-zero linearly independent vector fields which span $TM$.  These exist by the definition of a parallelisable manifold.  The construction of Corollary \ref{Cor:LocalGauging} can be extended globally using this choice of vector fields.
\end{proof}

Examples of parallelisable manifolds include Lie groups, $\sfS^1$, $\sfS^3$, $\sfS^7$, and all oriented three-manifolds.

We note that global gaugings may exist for manifolds which are not parallelisable. 

We see that the notion of local gauging allows us to get a local understanding of gauging with respect to regular smooth distributions. A regular smooth distribution has locally constant rank by definition.  Fix some $U\subset M$ such that $\dim(\text{Im}(\rho))=k$ is constant. We can construct a local gauging using the procedure described above (by definition there is a local set of linearly independent vector fields spanning the distribution).  The local gaugings may be of different dimensions for different choices of $U\subset M$. 

Consider rescaling a set of vector fields $\{\rho_a\}$ by $f\in C^\infty(M)$ to give $\{f\rho_a\}$.  Suppose that $\{\rho_a \}$ satisfies the generalised Killing condition
\begin{align*}
(\cL_{\rho_a}E)_{\mu\nu}=E_{\mu\lambda}\rho^\lambda_b(\Omega^+)^b{}_{\nu a}+E_{\lambda\nu}\rho^\lambda_b(\Omega^-)^b{}_{\mu a}.
\end{align*}
It follows that 
\begin{align*}
(\cL_{f\rho_a}E)_{\mu\nu}=E_{\mu\lambda}f\rho^\lambda_b(\Omega'^+)^b{}_{\nu a}+E_{\lambda\nu}f\rho^\lambda_b(\Omega'^-)^b{}_{\mu a},
\end{align*}
for $(\Omega'^\pm)^b{}_{\mu a}=(\Omega^\pm)^b{}_{\mu a}+\delta^b{}_{a}f^{-1}\partial_\mu f$.  If $f>0$ it is possible to rescale any set $\{\rho_a\}$ of vector fields whilst still satisfying the generalised Killing equations \eqref{Ecompatibility} by modifying the connection.  The generalised Killing equation is satisfied trivially if the vector fields $\{\rho_a \}$ are identically zero. We note that we cannot use a partition of unity type construction to globalise the vector fields used in the local gauged solutions by setting $\{\rho_a \}$ to zero outside of $U$.  It is clear that in the limit  $f\rightarrow 0$ we have $f^{-1}\partial_\mu f=\partial_\mu\ln(f)\rightarrow\infty$ and $\Omega'^\pm$ is unbounded.  There is no straightforward way to extend local gaugings to the entire manifold.
\section{Outlook and conclusion}\label{Sec:Outlook}
This paper examined a proposal for gauging non-linear sigma models with respect to a Lie algebroid action; integrability of local Lie algebroid actions to global Lie groupoid actions was also discussed.  The general conditions for gauging a non-linear sigma model with a set of involutive vector fields were given (Theorems \ref{Th:GaugeThrm}--\ref{Thrm:GaugeArbitrary}). It was shown that it is always possible to find a set of vector fields which will (locally) admit a Lie algebroid gauging.  Furthermore, the gauging process is not unique; if the vector fields span the tangent space of the manifold, there is a free choice of a flat connection (Corollary \ref{Thrm:FlatConnectionChoice}). 

The real restriction on the Lie algebroid gauging procedure is an understanding of when the gauged action is equivalent to the original action. In special cases it can be shown that the field $A$ is gauge equivalent to zero when it satisfies a Lie algebroid Maurer--Cartan equation (Section \ref{Sec:RecoverUngauged}). In these special cases the Lie algebroid is in fact locally equivalent to a Lie algebra and the proposal is locally equivalent to non-abelian gauging.  

The biggest open question surrounding Lie algebroid gauging is a better understanding of the gauge orbits of $A$.  It is essential to have a clear understanding of when $A$ is in the gauge orbit of zero and hence when the gauged action is equivalent to the ungauged action.  It is known that when $\Omega^+=\Omega^-$ the Lie algebroid gauging procedure is equivalent (locally) to non-abelian gauging.  It would be interesting to know if the more general Poisson--Lie T-duality could be incorporated in a Lie algebroid gauging based T-duality procedure when $\Omega^+\neq \Omega^-$. 
\section*{Acknowledgements}
KW would like to thank P. Bouwknegt, M. Bugden, and C. Klim\v{c}\'ik for many useful discussions on non-isometric T-duality which motivated the present work. This research was partially supported by the Australian Government through the Australian Research Council's Discovery Projects funding scheme (projects DP150100008 and DP160101520).  KW was also partially supported by the Mathematical Sciences Institute at the Australian National University through an MSI Kick-start fellowship.  

\appendix
\section{Proof of Lemma \ref{Lemma:Curvature}}\label{Sec:Appendix}
Using 
\begin{align}
(R_\nabla)^d{}_{abc}=2\rho^\mu_{[a|}\partial_\mu\Omega^d{}_{|b]c}+2\Omega^d{}_{[a|e}\Omega^e{}_{|b]c}-C^e{}_{ab}\Omega^d{}_{ec}.
\end{align}
and taking 
\begin{align}
({}^Q\Omega^+)^b{}_{ac}=(\sfE^{-1})^{bd}\rho^\mu_a\partial_\mu\sfE_{dc}-\sfE_{dc}({}^Q\Omega)^d{}_{ae}(\sfE^{-1})^{be}.
\end{align}

We wish to calculate the curvature of ${}^Q\nabla^+$ in terms of ${}^Q\Omega^-$. The connection coefficients ${}^Q\Omega^+$ consist of two terms
\begin{align*}
(\Omega_1)^b{}_{ac}=(\sfE^{-1})^{bd}\rho^\mu_a\partial_\mu\sfE_{dc},\quad (\Omega_2)^b{}_{ac}=-\sfE_{dc}({}^Q\Omega^-)^d{}_{ae}(\sfE^{-1})^{be}. 
\end{align*}  
Neither term individually defines a connection.
Note that 
\begin{subequations}\label{Eq:CurvSum}
	\begin{align}
	(R_{{}^Q\nabla^+})^d{}_{abc}=&2\rho^\mu_{[a|}\partial_\mu(\Omega_1)^d{}_{|b]c}+2(\Omega_1)^d{}_{[a|e}(\Omega_1)^e{}_{|b]c}-C^e{}_{ab}(\Omega_1)^d{}_{ec}\label{Eq:CurmSum1}\\
	&+2\rho^\mu_{[a|}\partial_\mu(\Omega_2)^d{}_{|b]c}+2(\Omega_2)^d{}_{[a|e}(\Omega_2)^e{}_{|b]c}-C^e{}_{ab}(\Omega_2)^d{}_{ec}\label{Eq:CurmSum2}\\
	&+2(\Omega_1)^d{}_{[a|e}(\Omega_2)^e{}_{|b]c}+2(\Omega_2)^d{}_{[a|e}(\Omega_1)^e{}_{|b]c}\label{Eq:CurmSum3}
	\end{align} 
\end{subequations}
We work this out line by line.  First consider \eqref{Eq:CurmSum1}:
\begin{align*}
2\rho^\mu_{[a|}\partial_\mu(\Omega_1)^d{}_{|b]c}=&2\rho^\mu_{[a|}\partial_\mu((\sfE^{-1})^{df}\rho^\nu_{|b]}\partial_\nu\sfE_{fc})\\
=&2\rho^\mu_{[a|}(\partial_\mu(\sfE^{-1})^{df})\rho^\nu_{|b]}\partial_\nu\sfE_{fc}+(\sfE^{-1})^{df}(2\rho^\mu_{[a|}\partial_\mu\rho^\nu_{|b]})\partial_\nu\sfE_{fc}\\
=&2\rho^\mu_{[a|}(\partial_\mu(\sfE^{-1})^{df})\sfE_{fg}(\sfE^{-1})^{gh}\rho^\nu_{|b]}\partial_\nu\sfE_{hc}+(\sfE^{-1})^{df}C^e{}_{ab}\rho^\nu_e\partial_\nu\sfE_{fc}\\
=&-2(\sfE^{-1})^{df}\rho^\mu_{[a|}(\partial_\mu\sfE_{fg})(\sfE^{-1})^{gh}\rho^\nu_{|b]}\partial_\nu\sfE_{hc}+C^e{}_{ab}(\sfE^{-1})^{df}\rho^\nu_e\partial_\nu\sfE_{fc}\\
=&-2(\Omega_1)^d{}_{[a|e}(\Omega_1)^e{}_{|b]c}+C^e{}_{ab}(\Omega_1)^d{}_{ec}.
\end{align*}
We conclude that \eqref{Eq:CurmSum1} is zero.

Next consider \eqref{Eq:CurmSum2}:
\begin{align*}
2\rho^\mu_{[a|}\partial_\mu(\Omega_2)^d{}_{|b]c}=&-2\rho^\mu_{[a|}\partial_\mu(\sfE_{ec}({}^Q\Omega^-)^e{}_{|b]f}(\sfE^{-1})^{df})\\
=&-2\rho^\mu_{[a|}(\partial_\mu\sfE_{ec})({}^Q\Omega^-)^e{}_{|b]f}(\sfE^{-1})^{df}-\sfE_{ec}(2\rho^\mu_{[a|}\partial_\mu({}^Q\Omega^-)^e{}_{|b]f})(\sfE^{-1})^{df}\\
&+2(\sfE_{ec})({}^Q\Omega^-)^e{}_{[b|f}\rho^\mu_{|a]}\partial_\mu(\sfE^{-1})^{df},
\end{align*}
\begin{align*}
2(\Omega_2)^d{}_{[a|e}(\Omega_2)^e{}_{|b]c}-C^e{}_{ab}(\Omega_2)^d{}_{ec}=&2\sfE_{ge}({}^Q\Omega^-)^g{}_{[a|f}(\sfE^{-1})^{df}\sfE_{ic}({}^Q\Omega^-)^i{}_{|b]h}(\sfE^{-1})^{eh}\\
&+C^e{}_{ab}\sfE_{gc}({}^Q\Omega^-)^g{}_{ef}(\sfE^{-1})^{df}\\
=&(\sfE^{-1})^{df}2({}^Q\Omega^-)^h{}_{[a|f}({}^Q\Omega^-)^i{}_{|b]h}\sfE_{ic}+C^e{}_{ab}\sfE_{gc}({}^Q\Omega^-)^g{}_{ef}(\sfE^{-1})^{df}\\
=&-\sfE_{ec}(2({}^Q\Omega^-)^e{}_{[a|h}({}^Q\Omega^-)^h{}_{|b]f}-C^h{}_{ab}({}^Q\Omega^-)^e{}_{hf})(\sfE^{-1})^{df}.
\end{align*}
We conclude that \eqref{Eq:CurmSum2} is equal to
\begin{subequations}
	\begin{align}
	\eqref{Eq:CurmSum2}=&-\sfE_{ec}(R_{{}^Q\nabla^-})^e{}_{abf}(\sfE^{-1})^{df}\\
	&-2\rho^\mu_{[a|}(\partial_\mu\sfE_{ec})({}^Q\Omega^-)^e{}_{|b]f}(\sfE^{-1})^{df}+2\sfE_{ec}({}^Q\Omega^-)^e{}_{[b|f}\rho^\mu_{|a]}\partial_\mu(\sfE^{-1})^{df}.\label{Eq:RemainderTerms}
	\end{align}
\end{subequations}
Now we calculate the terms of \eqref{Eq:CurmSum3}:
\begin{align*}
2(\Omega_1)^d{}_{[b|e}(\Omega_2)^e{}_{|a]c}=&-2(\sfE^{-1})^{df}\rho^\mu_{[b|}(\partial_\mu \sfE_{fe})\sfE_{hc}({}^Q\Omega^-)^h{}_{|a]g}(\sfE^{-1})^{eg}\\
=&2(\rho^\mu_{[b|}(\partial_\mu\sfE^{-1})^{df}) \sfE_{fe}\sfE_{hc}({}^Q\Omega^-)^h{}_{|a]g}(\sfE^{-1})^{eg}\\
=&2(\rho^\mu_{[b|}(\partial_\mu\sfE^{-1})^{df}) \sfE_{hc}({}^Q\Omega^-)^h{}_{|a]f}\\
=&-2\sfE_{ec}({}^Q\Omega^-)^e{}_{[a|f}\rho^\mu_{|b]}\partial_\mu(\sfE^{-1})^{df}.
\end{align*}
This cancels the second term of \eqref{Eq:RemainderTerms}.

\begin{align*}
2(\Omega_2)^d{}_{[a|e}(\Omega_1)^e{}_{|b]c}=&-2\sfE_{he}({}^Q\Omega^-)^h{}_{[a|f}(\sfE^{-1})^{df}(\sfE^{-1})^{eg}\rho^\mu_{|b]}\partial_\mu \sfE_{gc}\\
=&-2({}^Q\Omega^-)^h{}_{[a|f}(\sfE^{-1})^{df}\rho^\mu_{|b]}\partial_\mu \sfE_{hc}\\
=&2\rho^\mu_{[a|}(\partial_\mu \sfE_{ec})({}^Q\Omega^-)^e{}_{|b]f}(\sfE^{-1})^{df}.
\end{align*}
This term cancels with the first term of \eqref{Eq:RemainderTerms}.

Finally we conclude that
\begin{align}
(R_{{}^Q\nabla^+})^d{}_{abc}=-\sfE_{ec}(R_{{}^Q\nabla^-})^e{}_{abf}(\sfE^{-1})^{df}.
\end{align}


\end{document}